\documentclass[11pt, letterpaper]{amsart}
\usepackage{amsmath}
\usepackage{amssymb}
\usepackage[left=1in,right=1in,bottom=0.8in,top=0.8in]{geometry}
\usepackage{amsfonts}
\usepackage{graphicx}
\usepackage[font=small,labelfont=bf]{caption}
\usepackage{epstopdf}
\usepackage{hyperref}
\usepackage[dvipsnames]{xcolor}
\usepackage{amsthm}
\usepackage{float}
\usepackage{pgfplots}
\usepackage{listings}
\usepackage{longtable}
\usepackage{mathrsfs}
\usepackage{bbm}
\usepackage{euscript}
\usepackage{mathtools}
\usepackage{tikz}
\usepackage{comment}
\usepackage{tikz-cd}
\usepackage{amsxtra}
\usepackage{pifont}
\usepackage{amsbsy}
\usepackage{wasysym}
\usepackage{calligra}
\usepackage{url}
\usepackage{enumerate}
\usepackage[
backend=biber,
style=alphabetic,
]{biblatex}

\makeatletter
\renewcommand{\tocsection}[3]{%
  \indentlabel{\@ifnotempty{#2}{\bfseries\ignorespaces#1 #2.\quad}}\bfseries#3}
  \renewcommand{\tocsubsection}[3]{%
  \indentlabel{\hspace{2.5em}\@ifnotempty{#2}{\ignorespaces#1 #2.\quad}}#3}
\makeatother

\newtheorem{theorem}{Theorem}[section]
\newtheorem{corollary}[theorem]{Corollary}
\newtheorem{proposition}[theorem]{Proposition}
\newtheorem{lemma}[theorem]{Lemma}

\theoremstyle{definition}
\newtheorem{definition}[theorem]{Definition}

\theoremstyle{remark}
\newtheorem{remark}[theorem]{Remark}

\definecolor{energy}{RGB}{114,0,172}
\definecolor{refblue}{RGB}{30,90,180}
\definecolor{freq}{RGB}{45,177,93}
\definecolor{spin}{RGB}{251,0,29}
\definecolor{signal}{RGB}{203,23,206}
\definecolor{circle}{RGB}{217,86,16}
\definecolor{average}{RGB}{203,23,206}

\colorlet{shadecolor}{gray!20}
\pgfplotsset{compat=1.9}

\usepgflibrary{fpu}

\newcommand{\RR}{\mathbb{R}}
\newcommand{\ZZ}{\mathbb{Z}}
\newcommand{\NN}{\mathbb{N}}
\newcommand{\KK}{\mathbb{K}}
\newcommand{\QQ}{\mathbb{Q}}
\newcommand{\CC}{\mathbb{C}}

\newcommand{\PP}{\mathbb{P}}

\newcommand{\DD}{\mathbb{D}}
\newcommand{\bA}{\mathbb{A}}

\newcommand{\cB}{\mathcal{B}}

\newcommand{\cD}{\mathcal{D}}
\newcommand{\cE}{\mathcal{E}}
\newcommand{\cF}{\mathcal{F}}

\newcommand{\cH}{\mathcal{H}}

\newcommand{\cL}{\mathcal{L}}

\newcommand{\cO}{\mathcal{O}}

\newcommand{\cT}{\mathcal{T}}

\newcommand{\sfA}{\mathsf A}
\newcommand{\sfF}{\mathsf F}

\newcommand{\an}{\operatorname{an}}

\newcommand{\Bl}{\operatorname{Bl}}

\newcommand{\codim}{\operatorname{codim}}

\newcommand{\End}{\operatorname{End}}

\newcommand{\cEnd}{\mathscr{E}\text{\kern -3pt {\calligra\large nd}}\,}
\newcommand{\cBun}{\mathscr{B}\text{\kern -3pt {\calligra\large un}}\,}
\newcommand{\cCoh}{\mathscr{C}\text{\kern -3pt {\calligra\large oh}}\,}

\newcommand{\Gr}{\operatorname{Gr}}

\newcommand{\cHom}{\mathscr{H}\text{\kern -3pt {\calligra\large om}}\,}
\newcommand{\id}{\operatorname{id}}

\newcommand{\im}{\operatorname{im}}

\newcommand{\Pic}{\operatorname{Pic}}

\newcommand{\pr}{\operatorname{pr}}

\newcommand{\red}{\operatorname{red}}

\newcommand{\rank}{\operatorname{rank}}

\newcommand{\ev}{\operatorname{ev}}

\newcommand{\Spec}{\operatorname{Spec}}
\newcommand{\Sym}{\operatorname{Sym}}

\newcommand{\tr}{\operatorname{Tr}}

\newcommand{\hdg}{\mathrm{Hdg}}

\newcommand{\Eu}{\operatorname{Eu}}

\author[Ting Gong]{Ting Gong}
\email{tgong2@uw.edu}

\author[Shitan Xu]{Shitan Xu}
\email{shitanxu@bicmr.pku.edu.cn}

\title{Irrationality of the first stabilization of a complex smooth cubic threefold}
\begin{document}
                                                           
\begin{abstract}
 We show that for every smooth complex cubic threefold $Y$, the fourfold $Y\times\PP^1$ is irrational. As an application, we construct a smooth stably rational complex Fano variety $X$ which is irrational, while $X^m$ is rational for every $m\geq2$ and $\Sym^2 X$ is rational.
\end{abstract}

\maketitle

\section{Introduction}
Rationality questions have a long history in algebraic geometry, and smooth complex cubic threefolds have played a central role in their development. The irrationality of every smooth complex cubic threefold is classical \cite{CG}. Stable rationality is subtler: very general cubic threefolds are not stably rational \cite{EFS}, while Tschinkel and Zhang recently constructed stably rational examples \cite[Theorem~1.3]{TZ}. We show that the first stabilization is nevertheless irrational for every smooth cubic threefold.

\begin{theorem}\label{thm:main}
For every smooth cubic threefold $Y\subset\PP^4_\CC$, the product $Y\times\PP^1$ is not rational.
\end{theorem}

Together with the construction of Tschinkel and Zhang \cite{TZ}, Theorem~\ref{thm:main} gives the following consequences.

\begin{corollary}
\label{cor:fiber product}
    There exists a smooth projective complex variety $X$ such that, for every $m\geq 2$, the $m$-fold self-product $X^m:=\underbrace{X\times\cdots\times X}_{m\text{ factors}}$ is rational, while $X$ itself is not rational.
\end{corollary}

\begin{corollary}\label{cor:applications}
    There exists a smooth stably rational complex Fano variety $X$ which is not rational, while $X^m$ is rational for every $m\geq2$ and $\Sym^2X$ is rational.
\end{corollary}

Our proof uses the analytic quantum connection directly. The zero factor of a cubic threefold has full rank twelve and Hodge-class rank two, as does the quantum connection of a genus-five curve. Following the local computation of \cite[Section~3]{Cai}, we distinguish them by the residue polynomials
$$P_{\mathrm{cub}}(T)=(T-1/6)(T-5/6),\qquad P_{\mathrm{curve}}(T)=(T-1/2)^2$$
on specified logarithmic lattices. Counting factors of the first type gives an integer $n_{\mathrm{cub}}(X)$ satisfying

$$ n_{\mathrm{cub}}(\Bl_ZX)  =n_{\mathrm{cub}}(X)+(r-1)n_{\mathrm{cub}}(Z), \qquad r=\codim_XZ\geq2$$
 $$n_{\mathrm{cub}}(S)=0\quad(\dim S\leq2),\qquad  n_{\mathrm{cub}}(Y\times\PP^1)=2.$$

The blowup formula follows from the local analytic comparison proved in Theorem~\ref{thm:local-comparison}, starting from Iritani's formal decomposition \cite{iritani2025quantumcohomologyblowups} and its Hodge compatibility \cite{iritani2026notesdecompositiontheoremblowups}. We verify the convergence domains, openness of the parameter map, and preservation of the full lattices directly. Surface classification gives the low-dimensional vanishing. Weak factorization then shows that $n_{\mathrm{cub}}$ is a birational invariant in dimension at most four, proving Theorem~\ref{thm:main}.

Section~\ref{sec:quantum} develops the local transport and logarithmic lattice constructions. Section~\ref{sec:nu} defines the invariant and proves its blowup formula. Section~\ref{sec:main-proof} contains the cubic calculation, surface exclusion, and proof of the main theorem. The applications are given in Section~\ref{sec:applications}.

\section*{Acknowledgements}

The authors thank the Tianyuan Mathematical Center in Kunming for its hospitality, where this work was carried out. TG thanks the organizers of the workshop \emph{Rationality and Algebraic Cycles} for the invitation and hospitality. TG also thanks BICMR for its hospitality during his visit. The authors also thank Evgeny Shinder, Zhiyu Tian, and Jan Lange for helpful discussions and suggestions that improved this work.

\section*{AI disclosure} 
During the development of this work, the authors used GPT-5.6 Sol through API calls as a research assistant. After being directed to specific examples and mathematical techniques, the model assisted in developing some of the arguments presented in this paper. The authors independently checked the mathematical arguments, rewrote the proofs, and developed the consequences presented here. AI tools were also used for grammatical corrections and improvements to the exposition. The authors take full responsibility for the content and correctness of the paper.

\section{Analytic quantum connections}\label{sec:quantum}
We briefly recall the setup for quantum cohomology, following the conventions of \cite{Guere,KKPY,HYZZ,iritani2026notesdecompositiontheoremblowups,iritani2025quantumcohomologyblowups}. We use the notation of Gu\'er\'e \cite[Sections~1 and~2.3]{Guere} for Hodge classes, quantum multiplication, and the associated spectral data. Let $X$ be a smooth projective complex variety.

\subsection{The analytic base}
For a connected smooth projective complex variety $X$, write
\[ H^*(X)^{\hdg}  =\bigoplus_p\bigl(H^{2p}(X,\QQ)\cap H^{p,p}(X)\bigr).\]
This is a rational vector space. Following \cite[Sections~1.1--1.2]{Guere}, choose graded bases $\cH=(\alpha_0,\ldots,\alpha_h)$ of $H^*(X)^{\hdg}$ and  $\cB=(\phi_0,\ldots,\phi_m)$ of $H^*(X,\QQ)$, with $\alpha_0=\phi_0=1_X$. Order $\cH$ so that $\alpha_1,\ldots,\alpha_{\rho(X)}$ span the divisor classes, where $\rho(X)=\dim N^1(X)_\QQ$. Set $d_i:=\deg\alpha_i,$ and $\deg T_i:=2-d_i.$

The Novikov monomial $Q^\beta$, for an effective integral curve class $\beta$, has degree $2c_1(X)\cdot\beta$. Following \cite[Definition~3]{Guere}, write $R^*(X,\QQ)=\QQ[[Q,T_0,\ldots,T_h]]$ for the corresponding graded completion. If $(\phi^k)$ is the Poincar\'e dual basis, for the homogeneous element $\tau=\sum_{i=0}^h T_i\alpha_i$ of degree $2$, the quantum product is defined by \cite[Definition~6]{Guere}
\[\phi_i\star_\tau\phi_j =\sum_{\beta,n,k}\frac{Q^\beta}{n!} \langle\phi_i,\phi_j,\phi_k,\tau,\ldots,\tau\rangle^X_{0,3+n,\beta}\phi^k.\]
More generally, this extends to homogeneous elements of degree $2$ as in \cite[Remark~8]{Guere}. Let
$$\Eu_\tau=c_1(X)+\sum_i\left(1-\frac{d_i}{2}\right)T_i\alpha_i$$
denote the Euler vector field, and let $\kappa_\tau=\Eu_\tau\star_\tau$ denote quantum multiplication by the Euler vector field \cite[Definitions~9 and~10]{Guere}. 

Let $\KK=\widehat{\overline{\CC((s))}}$ be a complete non-Archimedean extension of $\CC$, with valuation $v(s)=1$ and trivial valuation on $\CC$. Set 
$$ H^*(X)^{\mathrm{Hdg}}_\KK =H^*(X)^{\mathrm{Hdg}}\otimes_\QQ\KK.$$
Throughout, $H^*(X,\KK)$ means $H^*(X,\QQ)\otimes_\QQ\KK$.
 
The following standard finiteness statement explains the numerical Novikov coordinates used below.

\begin{lemma}\label{lem:finite}
There are only finitely many integral homology classes of effective curve cycles of bounded degree with respect to a fixed integral ample class. Consequently, grouping Novikov series by numerical curve class is well-defined and respects products and cohomological grading.
\end{lemma}
\begin{proof}
After replacing the ample class by a positive multiple, assume it is very ample. Effective curve cycles of bounded degree are parametrized by a finite union of projective Chow varieties. Each has finitely many connected components, and the integral homology class is constant on each component. This proves the finiteness assertion.

A fixed numerical curve class has fixed ample degree, so only finitely many effective homology classes contribute to it. Addition and intersection with $c_1(X)$ factor through numerical equivalence, which proves the remaining assertions.
\end{proof}

Let $\Gamma_X$ be the lattice of integral numerical curve classes and
$T_X=(\Spec\KK[\Gamma_X])^{\an}$ its analytic torus. For a rigid point $Q\in T_X$, define $v_X(Q)\in N^1(X)_\RR$ by $v_X(Q)\cdot\beta=v(Q^\beta).$ Set
$$B_X=\{Q\in T_X:v_X(Q)\in\operatorname{Amp}(X)_\RR\} \times(\bA^1_\KK)^{\an} \times\DD_{<1}^{h-\rho(X)}.$$

Here the affine coordinate is $T_0$, while the coordinates on $\DD_{<1}^{h-\rho(X)}$ are the $T_i$ with $d_i\geq 4$. This analytic construction appears in \cite[Section~2.2.3]{HYZZ} and \cite[Section~3.5.2.1]{KKPY}. By construction, $B_X$ is smooth and connected. The Gromov--Witten potential, and hence the coefficients of the quantum product, define analytic functions on $B_X$ by \cite[Lemma~2.19]{HYZZ} and \cite[Lemma~3.29]{KKPY}.

Our base $B_X$ is the Hodge-class restriction of the maximal $\mathsf A$-model base, written in numerical Novikov coordinates with the divisor parameters absorbed. Here the Hodge-class restriction is the fixed locus of the Hodge group, whose fixed subspace in cohomology is $H^*(X)^{\mathrm{Hdg}}_\KK$; see \cite[Section~3]{iritani2026notesdecompositiontheoremblowups}. We retain the full quantum connection over this restricted base, together with its Hodge-class subconnection.

\subsection{Rigid points and evaluation}
We use Gu\'er\'e's evaluation conventions for explicit computations; geometrically, these are fibers of the quantum connection \cite[Definition~26]{Guere}.

Recall that a \emph{rigid point} $x\in B_X$ is represented on an affinoid chart $W=\operatorname{Sp}R$ by a maximal ideal $\mathfrak m_x$ with residue field finite over $\KK$ \cite[Definition~4.1.1]{Piotr}. Since $\KK$ is algebraically closed, it determines an evaluation map
$$\ev_x:R\longrightarrow\KK,\qquad f\longmapsto f(x).$$
Define the fiberwise spectrum and generalized eigenspace by
$$\mathrm{Sp}^X_{\ev_x}  =\{\alpha\in\KK:\det(\ev_x(\kappa_\tau)-\alpha I)=0\},\qquad  E^X_{\ev_x,\alpha}  =\ker(\ev_x(\kappa_\tau)-\alpha I)^M,\quad M\gg0.$$
We also write $\mathrm{Sp}^X_x$, $E^X_{x,\alpha}$, or omit the subscript $x$ on $\ev$ when the point is understood. This agrees with the normalization $b=1$ in \cite[Definition~26]{Guere}. Only the convergent quantum coefficients and the polynomial unit term are evaluated; no evaluation of the entire formal coefficient ring is required.

As in \cite[Definition~28]{Guere}, set
$$\rho^X_{\ev_x,\alpha} =\dim_\KK\bigl(E^X_{\ev_x,\alpha}\cap H^*(X)^{\mathrm{Hdg}}_\KK\bigr),\qquad \gamma^X_{\ev_x,\alpha} =\rank\bigl((\ev_x(\kappa_\tau)-\alpha I)|_{E^X_{\ev_x,\alpha}}\bigr).$$
We use the same abbreviation $\rho^X_{x,\alpha}=\rho^X_{\ev_x,\alpha}$. All ranks and traces are ordinary.

\subsection{Connection and horizontal factors}
Write $\cH_X$ for the trivial vector bundle over $B_X\times\DD_u$ with fiber $H^*(X,\KK)$, and $\cH_X^{\mathrm{Hdg}}$ for its subbundle with fiber $H^*(X)^{\mathrm{Hdg}}_\KK$. Here $\DD_u$ is the germ at $0$ of an analytic disc with parameter $u$. On $\cH_X$ we take the meromorphic quantum connection \cite[Section~3.5.2]{KKPY}, \cite[Section~2.3]{iritani2025quantumcohomologyblowups}.
$$\nabla_{\partial_u} =\partial_u-u^{-2}\kappa_\tau+u^{-1}\Gr_X, \qquad \nabla_{\partial_{T_i}} =\partial_{T_i}+u^{-1}(\alpha_i\star_\tau),$$
where the second formula is for $i=0$ or $d_i\geq4$, and $\Gr_X|_{H^k(X,\KK)}=\tfrac{k-\dim X}{2}I$.

For $D\in N^1(X)_\KK$, let $\xi_D$ be the translation-invariant vector field on $T_X$ defined by $\xi_D(Q^\beta)=(D\cdot\beta)Q^\beta.$ The connection in the Novikov-torus directions is $ \nabla_{\xi_D} =\xi_D+u^{-1}(D\star_\tau).$ Thus the divisor directions $\alpha_1,\ldots,\alpha_{\rho(X)}$ are encoded by the Novikov coordinates. Flatness follows from the formal quantum identities on the convergence domain. Moreover, $\cH_X^{\mathrm{Hdg}}$ is preserved by the connection; see \cite[Proposition~12]{Guere} and \cite[Lemma~7]{iritani2026notesdecompositiontheoremblowups}.

We consider the $u$-adic completion of $\cH_X$ and $\cH_X^{\mathrm{Hdg}}$:
\[\cL_X =(H^*(X,\KK)\otimes_\KK\cO_{B_X})[[u]],\qquad \cL_X^{\mathrm{Hdg}} =(H^*(X)^{\mathrm{Hdg}}_\KK\otimes_\KK\cO_{B_X})[[u]].\]
Put $\cE_X=\cL_X/u\cL_X=\cH_X|_{u=0}$.

\begin{definition}
    A \emph{horizontal projector} of the connection on an affinoid open chart $W=\operatorname{Sp}R\subset B_X$ is an endomorphism
$P(u):\cL_X|_W\to \cL_X|_W$ over $\cO_W[[u]]$ satisfying
\[
 P(u)^2=P(u),\qquad
 [\nabla_{\partial_u},P(u)]=0,\qquad
 [\nabla_\xi,P(u)]=0\quad(\xi\in TW).
\]
\end{definition}

We are interested in the spectral projectors of $\kappa_\tau$ on the residual bundle $\cE_X$. The following lemma lifts a projector onto a cluster of generalized eigenspaces when its spectrum is disjoint from the complementary spectrum.

\begin{lemma}\label{lem:separation}
Let $P_0:\cE_X|_W\to\cE_X|_W$ be a spectral projector of $\kappa_\tau$
on a neighborhood $W\subset B_X$, with selected and complementary
spectra disjoint. After shrinking $W$ and the $u$-disc, it lifts to a
unique horizontal analytic projector $P(u)$ with $P(0)=P_0$.
The decomposition preserves $\cH_X^{\hdg}$.
\end{lemma}

\begin{proof}
Apply \cite[Proposition~3.36]{HYZZ} to
\[
 \cE_X|_W=\im P_0\oplus\ker P_0,
 \qquad
 K=\left.\nabla_{u^2\partial_u}\right|_{u=0}
   =-\kappa_\tau.
\]
This gives an analytic splitting preserved by the vertical
connection after shrinking. Its formal completion is also
preserved by every base-direction operator by
\cite[Proposition~3.29]{HYZZ}; hence the analytic splitting
is horizontal.

For uniqueness and Hodge compatibility, write locally
$$\nabla_{\partial_u} =\partial_u+u^{-2} \begin{pmatrix}
 A_1(u)&0\\
 0&A_2(u) \end{pmatrix}.$$
A horizontal morphism $D(u)$ between summands with disjoint leading spectra satisfies
\[
 u^2D'+A_1D-DA_2=0.
\]
If $D=u^mD_m+O(u^{m+1})$ with $D_m\ne0$, then
\[
 A_1(0)D_m-D_mA_2(0)=0,
\]
contrary to \cite[Lemma~3.25]{HYZZ}. Thus such morphisms vanish.
Applied to the identity between two lifted splittings, this
proves uniqueness. Applied to the inclusion of the Hodge-class
subbundle, after splitting that subbundle by the same theorem,
it proves compatibility with $\cH_X^{\hdg}$.
\end{proof}
\begin{definition}
\begin{itemize}
    \item A \emph{separated factor} is the image of a horizontal projector whose reduction at $u=0$ is a spectral projector with disjoint spectra on its image and kernel. It is \emph{primary} at a point if its leading endomorphism has only one eigenvalue there.
    \item   The \emph{lattice} of a separated factor is the image of its projector on $\cL_X$, and its Hodge-class lattice is the image on $\cL_X^{\mathrm{Hdg}}$. 
    \item A \emph{ regular lattice isomorphism} is an isomorphism over $\cO[[u]]$ commutes with the connections. When working analytically, we require it and its inverse to converge near $u=0$.
\end{itemize}

\end{definition}

Recall that $\nabla_{u^2\partial_u}|_{u=0}=-\kappa_\tau$ and set $\mu(\xi)=\nabla_{u\xi}|_{u=0}$. Then the unit $1_X$ gives an isomorphism:
$$TB_X\xrightarrow{\sim} \cL_X^{\mathrm{Hdg}}/u\cL_X^{\mathrm{Hdg}}, \qquad \xi\longmapsto\mu(\xi)1_X.$$
    
Indeed, 
$$\mu(\partial_{T_i})1_X=\alpha_i \quad (i=0\ \text{or}\ d_i\geq4),\qquad \mu(\xi_{\alpha_i})1_X=\alpha_i \quad (1\leq i\leq\rho(X)),$$
and these are precisely the $h+1$ elements of the basis $\cH$.

Both $\kappa_\tau$ and $\mu(\xi)$ are multiplication operators by elements of the unital even algebra $H^*(X)^{\mathrm{Hdg}}_\KK$. For multiplication $m_\eta$ by an element $\eta$ of this algebra and any polynomial $f$, one has
$$f(m_\eta)|_{H^*(X,\KK)}=0 \quad\Longleftrightarrow\quad f(\eta)=0 \quad\Longleftrightarrow\quad f(m_\eta)|_{H^*(X)^{\mathrm{Hdg}}_\KK}=0,$$
by evaluation at $1_X$; compare \cite[Lemma~5.19]{KKPY}. Thus the minimal polynomials, and in particular the reduced spectra, of the full and Hodge-class actions agree. Their primary projectors are multiplication by the same idempotents.
 
We also use external direct sums, namely sums of connections pulled back to the product of their bases \cite[Definition~3.3]{KKPY}. The same argument applies to the resulting direct-product algebra and its unit. 

\begin{lemma}\label{lem:divisor}
Fix a character $\chi_*:\Gamma_X\to\KK^\times$ whose valuation class $v_X(Q_*)\in N^1(X)_\RR$ is ample, and let $x_*=(Q_*,T'_*)\in B_X$ be a rigid point with this character. All Hodge-class parameters, including the divisor parameters, give an analytic coordinate chart from a neighborhood of $(0,T'_*)$ onto a neighborhood of $x_*$. The coordinate change leaves $u$ fixed and identifies the quantum connections, and their Hodge-class subbundles.
\end{lemma}

\begin{proof}
Put $r=\rho(X)$ and write
\[\delta=\sum_{i=1}^rT_i\alpha_i,\qquad \tau=\tau'+\delta,\qquad  T=(T_1,\ldots,T_r),\quad T'=(T_0,T_{r+1},\ldots,T_h).\]
By the divisor equation \cite[Section~1.3 and Remarks~7--8]{Guere}, we have 
$$\star_{\chi_*,\,\tau'+\delta}  =\star_{\chi_\delta,\,\tau'}.$$
Here $\chi_\delta(\beta)=\chi_*(\beta)\exp(\delta\cdot\beta)$. $\chi_\delta$ is indeed a character by additivity of $\delta\cdot\beta$. Choose an integral basis $\{\beta_j\}$ of $\Gamma_X$ and let $M_{ij}=\alpha_i\cdot\beta_j, Q_j^*=\chi_*(\beta_j).$ Then $\chi_{\delta}$ is determined by its value at the basis $\{\beta_j\}$ with coordinate function as $q^{\delta}_j:=\chi_{\delta}(\beta_j)=Q_j^*\exp(\sum_iM_{ij}T_i)$. The numerical intersection pairing makes $M$ invertible. For sufficiently small $\delta$, the exponentials converge and have norm one, so the Novikov valuation class is unchanged as 
$$v(\chi_{\delta}(\beta))=v(\chi_*(\beta))+v(\exp(\delta\cdot\beta))=v(\chi_*(\beta)).$$
Moreover, Let $\sfA$ denote the change-of-coordinate matrix $\sfA:=\sfA(T,T')=(\{q^{\delta}_j\}_j,T')$, we have 
\[\det d\sfA_{(0,T'_*)} =\left(\prod_jQ_j^*\right)\det M\ne0.\]
Thus $\sfA$ is an isomorphism of analytic neighborhoods after shrinking. For the connection, the chain rule gives
\[
 \sfA_*\partial_{T_i}=\sum_jM_{ij}q^{\delta}_j\partial_{q^{\delta}_j}=\xi_{D_i},
 \qquad
 (\sfA^*\nabla)_{\partial_{T_i}}
      =\partial_{T_i}+u^{-1}(D_i\star_{\chi_*,\tau'+\delta}).
\]
The other parameter coordinates are unchanged. Since the divisor
coefficient $1-d_i/2$ in $\Eu_\tau$ is zero,
\[
 \Eu_{\tau'+\delta}=\Eu_{\tau'},\qquad
 \kappa_{\chi_*,\tau'+\delta}=\kappa_{\chi_\delta,\tau'},\qquad
 (\sfA\times\id_{\DD_u})^*\nabla
        =\nabla.
\]
So $\sfA$ preserves Hodge-class subbundle, with no affection on $u$.

For a fixed complex divisor shift $\delta_0$, apply the same argument to $\delta-\delta_0$ also gives an isomorphism between a neighborhood of $x_*$ with a neighborhood of $(\delta_0,T'_*)$.
\end{proof}

We next show that a nonsplit rank 2 generalized eigenspace of $\kappa_\tau$ persist in a small neighborhood, this is claimed in \cite[Remark~27 and Example~29]{Guere} and \cite[Section~3.3 and Remark~3.14]{KKPY}. We provide a detailed proof in our special rank 2 case. Besides, we freely replace the base $B_X$ by an appropriate open subset in the following proposition. Let $P_0:\cE_X\to\cE_X$ be a spectral projector of $\kappa_\tau$. By Lemma \ref{lem:separation}, there exists a horizontal projector $P(U):\cL_X\to \cL_X$ lifts $P_0$ after properly shrinking $B_X$. Let $\cF=P(U)\cL_X$, and $L=P(U)\cL_X^{\hdg}$. 

\begin{proposition}\label{prop:transport}
Suppose $L$ has rank 2
and, at a rigid point $x_0\in B_X$, the endomorphism induced by the Euler vector field of $(L/uL)_{x_0}$ is given by a nonsplit rank 2 Jordan block:
\[
 \kappa_{\tau,L}(x_0):=\ev_{x_0}(\kappa_\tau)|_{(L/uL)_{x_0}}=\alpha_0I+N_0,
 \qquad N_0^2=0,\quad N_0\ne0.
\]
Let $\alpha(x)=\tfrac12\tr\kappa_{\tau,L}(x)$.
Then there is an open neighborhood $W\subset B_X$ of $x_0$, such that every rigid $x\in W$ satisfies
\[
 \kappa_{\tau,L}(x)=\alpha(x)I+N_x,
 \qquad N_x^2=0,\quad N_x\ne0.
\]
Futhermore, for $\cF_x=\cF|_{\{x\}\times\DD_u}$, consider the restriction $\nabla_{x,\partial_u}:=\nabla_{\partial_u}|_{\{x\}\times\DD_u}$, define
\[ 
\nabla^0_{x,\partial_u}   
=\nabla_{x,\partial_u}+\frac{\alpha(x)}{u^2}I.
\]

After shrinking the $u$-disc, for every rigid $x\in W$ there exist
$c_x\in\KK^\times$, independent of $u$, and an isomorphism
\[
 (\cF_x,\nabla_x^0)
 \simeq h_{c_x}^*(\cF_{x_0},\nabla_{x_0}^0),
 \qquad h_{c_x}(u)=c_xu.
\]
The isomorphism and its inverse are analytic and regular at $u=0$,
and identify the Hodge-class subbundles.
\end{proposition}
\begin{proof}
For simplicity let $P=P_0$. Let $\kappa_{\tau,\cF}$ be the
restriction of $\kappa_\tau$ to $\cF/u\cF$. Define the vector fields $e_P, \Eu_P$ and the set $\cD$ by:
\[
 \begin{gathered}
 \mu(e_P)1_X=P1_X,\qquad
 \mu(\mathsf{Eu}_P)1_X=P\kappa_\tau1_X, \qquad
 \cD=\{\xi:\mu(\xi)1_X\in(1-P)H^*(X)^{\hdg}_\KK\}.
 \end{gathered}
\]
Here these equations define vector fields because
$\xi\mapsto\mu(\xi)1_X$ is an isomorphism. Thus $P1_X$ is the unit of $L/uL$, and
multiplication by $P\kappa_\tau1_X$ on this algebra is
$\kappa_{\tau,L}$. Elements of $\cD$ annihilate
$L/uL$, since $P(1-P)=0$. Hence the leading actions on
the selected full factor are
\[
 \mu(e_P)=I,\qquad
 \mu(\mathsf{Eu}_P)=\kappa_{\tau,L},\qquad
 \mu(\cD)=0.
\]
In particular, $\kappa_{\tau,L}$ would be a scalar action if $\Eu_P$ and $e_p$ are linearly dependent. This contradicts to our assumption $N_0\neq 0$. Since being linearly independent is an open condition, and the selected Hodge-class algebra $L$ has rank two, $\Eu_P$ and $e_p$ form a basis after shrinking $W$, and we have
$$TW=\cO_W e_P\oplus\cO_W\mathsf{Eu}_P\oplus\cD.$$

Choose an local frame of $\cF$, according to our definition of connection, write:
\[
 \nabla_{\partial_u}=\partial_u+A(x,u),\qquad
 \nabla_\xi=\xi+B_\xi(x,u).
\]
where:
\[
 A=A(x,u)=-\frac{\kappa_{\tau,F}(x)}{u^2}+O(u^{-1}),\qquad
 B_{e_P}=\frac Iu+O(1),\]
\[ B_{\mathsf{Eu}_P}=\frac{\kappa_{\tau,F}(x)}u+O(1),\qquad
 B_\xi=O(1)\quad(\xi\in\cD).
\]
Consequently, $B_{e_P}-I/u, B_{\mathsf{Eu}_P}+uA$, and 
 $B_\xi$ are regular along $u=0$.

To compare the local behavior at nearby fibers of $x_0$, we study
these flows case by case. Let $t$ denote the flow parameter, usually called the time parameter. 

\begin{itemize}
    \item The flow $x(t)$ along $e_P$ is defined by solving: 
    $$ \qquad x'(t)=e_P(x(t)),\qquad x(0)=x_0.$$
Let $\Phi:\DD_t\times \DD_u\to W\times \DD_u$ be the map:$(t,u)\mapsto (x(t),u)$. 
Then we have 
$$(\Phi^*\nabla)_{\partial_u}=\partial_u+A(x(t),u), \qquad (\Phi^*\nabla)_{\partial_t}=\partial_t+B_{e_P}(x(t),u).$$
We then check that there exists a frame such that the connection has the local form in the statement.  For the $t$ direction, denote the difference by 
$$C(t,u):=B_{e_P}(x(t),u)-\frac Iu$$
Then $C(t,u)$ is regular along $u=0$, and we have $(\Phi^*\nabla)_{\partial_t}=\partial_t+\frac Iu+C(t,u)$.

We claim that there exists a frame such that $C(t,u)$ can be removed, and then the local connections are preserved. In fact, consider the following equations
$$\partial_tG(t,u)=-C(t,u)G(t,u),\qquad G(0,u)=I.$$
write $C(t,u)=\sum C_i(u)t^i$ and $G(t,u)=\sum G_j(u)t^j$,
the initial-value equation gives
\[
 (j+1)G_{j+1}(u)=-\sum_{i=0}^jC_i(u)G_{j-i}(u),\qquad G_0=I.
\]
Hence there exists a regular $G(t,u)$ that can be inductively solved from the above system of equations. The convergency of $G(t,u)$ is guaranteed by the equations either. Moreover, one check directly that $G(t,u)$ is invertible by having 
$$G^{-1}(t,u)=\sum^{\infty}_{n=0}(I-G(t,u))^n.$$ 
$G^{-1}(x,u)$ is the new frame that would remove $C(t,u)$. Indeed, we have 
 \[
 \begin{aligned}
     &G^{-1}(t,u)(\partial_t+I/u+C(t,u))G(t,u)\\&=\partial_t+I/u+C(t,u)+G^{-1}(t,u)\frac{\partial G(t,u)}{\partial t}\\&=\partial_t+I/u+C(t,u)-C(t,u)\\ &=\partial_t+I/u
 \end{aligned}
 \]
 We also need to check the changes of connection along $u$ direction under the new frame, 
\[
 \begin{aligned}
     &G^{-1}(t,u)(\partial_u+A(x(t),u))G(t,u)\\&=\partial_u+G^{-1}(t,u)A(x(t),u)G(t,u)+G^{-1}(t,u)\frac{\partial G(t,u)}{\partial u}
 \end{aligned}
 \]
Let $\widehat{A}(t,u):=G^{-1}(t,u)A(x(t),u)G(t,u)+G^{-1}(t,u)\frac{\partial G(t,u)}{\partial u}$ be the new local matrix.  Since flatness of connection is preseved in a new frame, under the conjugation action of $G(t,u)$, we still have 
\[
 \begin{aligned}
 0&=[\partial_t+I/u,\partial_u+\widehat{A}(t,u)]\\
     &=\frac{\partial\widehat{A}(t,u)}{\partial t}-\frac{\partial}{\partial u}(\frac{I}{u})+[\frac{I}{u},\widehat{A}(t,u)]\\
     &=\frac{\partial\widehat{A}(t,u)}{\partial t}+\frac{I}{u^2}
 \end{aligned}
 \]
 Solve the above equation with respect to $t$ shows that $\widehat{A}(t,u)=\widehat{A}(0,u)-\frac{tI}{u^2}$. Note also $\widehat{A}(0,u)=A(x_0,u)$ by definition, hence we have 
 $$\widehat{A}(t,u)=A(x_0,u)-\frac{t}{u^2}I.$$
 Now taking the coefficients of the $u^{-2}$ term on both sides, we have 
 $$G(t,0)^{-1}\kappa_{\tau,L}(x(t))G(t,0)=\kappa_{\tau,L}(x_0)+tI$$
 Taking trace on both side and use $\alpha(x)=\tfrac12\tr\kappa_{\tau,L}(x)$, we have 
 $$\alpha(x(t))=\alpha(x_0)+t$$
 Hence 
 \[
 \begin{aligned}
 G(t,u)^{-1}\nabla^0_{x(t),\partial_u}G(t,u)
 &=G(t,u)^{-1}(\partial_u+A(x(t),u)+\frac{\alpha(x(t))}{u^2}I)G(t,u)\\
     &=\partial_u+A(x_0,u)-\frac{t}{u^2}I+\frac{\alpha(x(t))}{u^2}I\\
     &=\partial_u+A(x_0,u)+\frac{\alpha(x_0)}{u^2}I\\
     &=\nabla^0_{x_0,\partial_u}
 \end{aligned}
 \]
Thus the normalized vertical connections $\nabla^0_{x,\partial_u}$ are regularly isomorphic.

\item The flow $x(t)$ along $\Eu_P$ is defined by solving
\[
 x'(t)=\Eu_P(x(t)),\qquad x(0)=x_0.
\]
In this case we also rescale the $u$ coordinate. Let
\[
 \Phi:\DD_t\times\DD_u\longrightarrow W\times\DD_u,
 \qquad \Phi(t,u)=(x(t),U=e^tu).
\]
Take $\DD_t$ small enough that $e^t$ converges. Since $dU=e^tdu+Udt$, the pullback operators are
\[
 \begin{aligned}
 (\Phi^*\nabla)_{\partial_u}
   &=\partial_u+e^tA(x(t),e^tu),\\
 (\Phi^*\nabla)_{\partial_t}
   &=\partial_t+B_{\Eu_P}(x(t),e^tu)
                 +e^tuA(x(t),e^tu).
 \end{aligned}
\]
Let
\[
 C(t,u):=B_{\Eu_P}(x(t),e^tu)+e^tuA(x(t),e^tu).
\]
Then $C(t,u)$ is regular along $u=0$ as the leading terms cancel:
\[
 C(t,u)=\frac{\kappa_{\tau,\cF}(x(t))}{e^tu}
       -\frac{\kappa_{\tau,\cF}(x(t))}{e^tu}+O(1)=O(1).
\]
 Solve again
\[
 \partial_tG(t,u)=-C(t,u)G(t,u),\qquad G(0,u)=I.
\]
As in the unit case, $G(t,u)$ and $G(t,u)^{-1}$ are analytic,
regular at $u=0$, and preserve the Hodge-class subbundle. Then the $t$ direction becomes
\[
 \begin{aligned}
 G(t,u)^{-1}(\partial_t+C(t,u))G(t,u)
 &=\partial_t+G(t,u)^{-1}C(t,u)G(t,u)
     +G(t,u)^{-1}\frac{\partial G(t,u)}{\partial t}\\
 &=\partial_t.
 \end{aligned}
\]
For the $u$ direction, set
\[
 \widehat A(t,u)
 :=G(t,u)^{-1}e^tA(x(t),e^tu)G(t,u)
   +G(t,u)^{-1}\frac{\partial G(t,u)}{\partial u}.
\]
Flatness now gives
\[
 0=[\partial_t,\partial_u+\widehat A(t,u)]
   =\frac{\partial\widehat A(t,u)}{\partial t}.
\]
Since $G(0,u)=I$ and $e^0=1$, we have
$\widehat A(0,u)=A(x_0,u)$, and hence
\[
 \widehat A(t,u)=A(x_0,u).
\]
Comparing the coefficients of $u^{-2}$ on both sides
gives
\[
 G(t,0)^{-1}e^{-t}\kappa_{\tau,\cF}(x(t))G(t,0)
 =\kappa_{\tau,\cF}(x_0).
\]
Restricting to the Hodge-class part and taking half the trace yields
\[
 \alpha(x(t))=e^t\alpha(x_0).
\]
The scalar normalization is compatible with the pullback, because
\[
 e^t\frac{\alpha(x(t))}{(e^tu)^2}I
 =\frac{\alpha(x_0)}{u^2}I.
\]
Therefore
\[
 \begin{aligned}
 &G(t,u)^{-1}\left(
 \partial_u+e^tA(x(t),e^tu)
       +e^t\frac{\alpha(x(t))}{(e^tu)^2}I\right)G(t,u)\\
 &\qquad=\partial_u+\widehat A(t,u)
                    +\frac{\alpha(x_0)}{u^2}I
 =\nabla^0_{x_0,\partial_u}.
 \end{aligned}
\]
In other words,
\[
 h_{e^t}^*(\cF_{x(t)},\nabla^0_{x(t)})
 \simeq(\cF_{x_0},\nabla^0_{x_0}).
\]

\item For a local section $\xi\in\cD$, the flow $x(t)$ is defined
by solving
\[
 x'(t)=\xi(x(t)),\qquad x(0)=x_0.
\]
Keep $u$ fixed and let $\Phi(t,u)=(x(t),u)$. Then
\[
 (\Phi^*\nabla)_{\partial_u}=\partial_u+A(x(t),u),\qquad
 (\Phi^*\nabla)_{\partial_t}=\partial_t+B_\xi(x(t),u).
\]
Here $C(t,u):=B_\xi(x(t),u)$ is already regular along $u=0$.
Solving
\[
 \partial_tG(t,u)=-C(t,u)G(t,u),\qquad G(0,u)=I
\]
therefore makes the time operator equal to $\partial_t$. Put
\[
 \widehat A(t,u)
 :=G(t,u)^{-1}A(x(t),u)G(t,u)
   +G(t,u)^{-1}\frac{\partial G(t,u)}{\partial u}.
\]
Exactly as in the Euler case, flatness gives
\[
 0=[\partial_t,\partial_u+\widehat A(t,u)]
   =\frac{\partial\widehat A(t,u)}{\partial t},\qquad
 \widehat A(t,u)=A(x_0,u).
\]
Comparing leading coefficients and taking half the trace on the
Hodge-class part gives
\[
 G(t,0)^{-1}\kappa_{\tau,\cF}(x(t))G(t,0)
 =\kappa_{\tau,\cF}(x_0),\qquad
 \alpha(x(t))=\alpha(x_0).
\]
Consequently,
\[
 \begin{aligned}
 G(t,u)^{-1}\nabla^0_{x(t),\partial_u}G(t,u)
 &=\partial_u+\widehat A(t,u)+\frac{\alpha(x(t))}{u^2}I\\
 &=\nabla^0_{x_0,\partial_u}.
 \end{aligned}
\]
These comparisons are regular and preserve the Hodge-class
subbundle.
\end{itemize}

To summarize, along a $e_P$ flow, $\kappa_{\tau,\cF}$ is changed by
conjugacy and scalar translation; along an Euler-flow it
is changed by conjugacy and multiplication by $e^t$; along a
flow for some vector field in $\cD$, it is changed only by conjugacy. After restricting to $L$, subtracting $\alpha(x)I$ cancels the translation.
Since each nearby fiber can be reached through a combination of these three type of flows. We have shown that at every rigid $x\in W$,
\[
 N_x:=\kappa_{\tau,L}(x)-\alpha(x)I
\]
is a nonzero scalar multiple of a conjugate of $N_0$. Hence
$N_x^2=0$ and $N_x\ne0$, proving our first assertion.

For the second claim, note that all three flow calculations were made on $\cF$,
and their changes of frame preserve the Hodge-class subbundle $L$.
Unit and complementary segments require no rescaling of $u$,
whereas an Euler segment contributes the constant homothety
$u\mapsto e^{-t}u$. Composing the comparisons, with the frame
matrices pulled back by the preceding homotheties, gives
\[
 (\cF_x,\nabla_x^0)
 \simeq h_{c_x}^*(\cF_{x_0},\nabla_{x_0}^0),
 \qquad c_x\in\KK^\times.
\]
Here $c_x$ is the product of coefficients come from the $\Eu_P$ flow,
so it is independent of $u$. Finite products and constant
homotheties preserve analyticity and regularity of the frame
matrices and their inverses. They also preserve the Hodge-class
subbundles. Since $h_{c_x}$ preserves the ideal $(u)$, completion
identifies the $\cF$ and their Hodge-class
sublattices $L$. This proves the second assertion.
\end{proof}

\subsection{A logarithmic lattice}\label{sec:residue}
Let $L$ be as in Proposition~\ref{prop:transport}. Then there exists an open neighborhood $W\subset B_X$, such that $\kappa_{\tau,L}$ is nonscalar with a unique eigenvalue at every rigid point of $W$. Let $N=\kappa_{\tau,L}-\alpha I$, we know $N^2=0$ with $\rank N=1$. Let $\ell:=\ker N=\im N\subset L/uL$. Then $\ell$ is a rank 1 subbundle. 
\begin{definition}\label{def:logarithmic-lattice}
Define the \emph{elementary modification}
and the \emph{normalized vertical connection} by
$$ L^\sharp=\{v\in L:v\bmod u\in\ell\},\qquad
 \nabla^0_{\partial_u}=\nabla_{\partial_u}+\frac{\alpha}{u^2}I.$$
We call $L^\sharp$ \emph{logarithmic} for $\nabla^0$ if
\[
 \nabla^0_{u\partial_u}(L^\sharp)\subseteq L^\sharp,
 \qquad\iff\qquad
 \nabla^0_{\partial_u}(L^\sharp)\subseteq u^{-1}L^\sharp.
\]
In that case, we define the induced residue map $R^{\sharp}$ as 
\[
 R^\sharp:=(u\nabla^0_{\partial_u})\bmod u:
 L^\sharp/uL^\sharp\longrightarrow L^\sharp/uL^\sharp,
 \qquad [v]\longmapsto[u\nabla^0_{\partial_u}v].
\]
\end{definition}

\begin{lemma}\label{lem:residue-lattice}
$L^\sharp$ is a rank-two lattice, and the endomorphism
\[
 (u^2\nabla^0_{\partial_u})\bmod u:
 L^\sharp/uL^\sharp\longrightarrow L^\sharp/uL^\sharp
\]
vanishes precisely when $L^\sharp$ is logarithmic.
If $L^\sharp$ is logarithmic, the characteristic polynomial of
$R^\sharp$ are preserved by regular base changes of the normlized
vertical connections and by $h_{c_x}$ as in Proposition \ref{prop:transport}. Finally, the characteristic of $R^{\sharp}$
is locally constant also by Proposition~\ref{prop:transport}.
\end{lemma}

\begin{proof}
Write $\cO=\cO_W$. Choose a local frame of $L$ adapted to $\ell$:
\[
 L=\cO[[u]]e_1\oplus\cO[[u]]e_2,\qquad
 \ell=\cO\overline e_1,\qquad
 N=\begin{pmatrix}0&n\\0&0\end{pmatrix},\quad n\in\cO^\times.
\]
The definition gives
\[
 L^\sharp=\cO[[u]]e_1\oplus\cO[[u]](ue_2),\qquad
 uL\subseteq L^\sharp\subseteq L,\qquad
 L^\sharp[u^{-1}]=L[u^{-1}].
\]
Thus $L^\sharp$ has rank two. Notice that
\[
 L^\sharp/uL=\ell,\qquad
 L^\sharp/uL^\sharp
 =\cO[e_1]\oplus\cO[ue_2];
\]

In the frame $(e_1,e_2)$, write
\[
 \nabla^0_{\partial_u}=\partial_u+A^0(x,u),\qquad
 A^0(x,u)=-\frac1{u^2}\begin{pmatrix}0&n\\0&0\end{pmatrix}
          +\frac1u(a_{ij})+(b_{ij})+O(u),
\]
where $n,a_{ij},b_{ij}$ are analytic functions of $x$. The frame
$(e_1,ue_2)$ of $L^\sharp$ is obtained by conjugation with $S=\operatorname{diag}(1,u)$, hence:
$$A^\sharp(x,u)=S^{-1}A^0(x,u)S+S^{-1}\partial_uS=\frac1{u^2}\begin{pmatrix}0&0\\a_{21}&0\end{pmatrix}
  +\frac1u\begin{pmatrix}a_{11}&-n\\b_{21}&a_{22}+1\end{pmatrix}
  +O(1).$$
So we know $L^{\sharp}$ is logarithmic if and only if $a_{21}=0.$ For $v\in L^\sharp$, $u^2\nabla^0_{\partial_u}(uv)
   =u^2v+u^3\nabla^0_{\partial_u}v\in uL^\sharp.$
Hence $u^2\nabla^0_{\partial_u}$ descends to the quotient $L^\sharp/uL^\sharp$, where its
matrix is
\[
 (u^2\nabla^0_{\partial_u})\bmod u
   =\begin{pmatrix}0&0\\a_{21}&0\end{pmatrix}.
\]
So we have
$$(u^2\nabla^0_{\partial_u})\bmod u=0\iff a_{21}=0 \iff L \text{ is logarithmic.}$$
When L is logarithmic, the residue matrix $R^{\sharp}$ is
$$R^\sharp=\begin{pmatrix}a_{11}&-n\\b_{21}&a_{22}+1\end{pmatrix}.$$

Let $\Psi:L\xrightarrow{\sim}L'$ be a regular lattice isomorphism
intertwining the centered vertical connections. Comparing their
leading terms gives
\[
 \overline\Psi N=N'\overline\Psi,
 \qquad \overline\Psi:=\Psi\bmod u.
\]
It follows that
\[
 \overline\Psi(\ell)=\ell',\qquad
 \Psi(L^\sharp)=(L')^\sharp.
\]
Thus the induced map $\Psi^\sharp:L^\sharp\xrightarrow{\sim}(L')^\sharp$
shows that $a_{21}=0$ if and only if $a^{'}_{21}=0$. If they vanish, then
\[
 \overline{\Psi^\sharp}\,R^\sharp
   =(R')^\sharp\,\overline{\Psi^\sharp},\qquad
 \det(tI-R^\sharp)=\det(tI-(R')^\sharp).
\]
Under $h_c(u)=U=cu$, with $c\in\KK^\times$, the vertical matrix becomes
\[
 cA^0(x,cu)
  =-\frac{N/c}{u^2}+\frac{(a_{ij})}{u}+c(b_{ij})+O(u).
\]
Therefore
\[
 \begin{pmatrix}a_{11}&-n/c\\cb_{21}&a_{22}+1\end{pmatrix}
 =\operatorname{diag}(1,c)\,R^\sharp\,\operatorname{diag}(1,c^{-1}).
\]
Hence the charactertic polynomial does not change either.

Finally, By proposition~\ref{prop:transport}, the same comparisons consequently give
\[
 (L_x^\sharp,\nabla_x^0)
 \simeq h_{c_x}^*(L_{x_0}^\sharp,\nabla_{x_0}^0),
\]
where the connections are understood on the localized lattices.
Thus, for every sufficiently nearby rigid $x$, $L_x^\sharp$ is logarithmic if and only if $L_{x_0}^\sharp$ is logarithmic. When these equivalent conditions hold,
\[
 \det(tI-R_x^\sharp)=\det(tI-R_{x_0}^\sharp).
\]
This proves our assertions.
\end{proof}

\section{Obstructions and blowup formula}\label{sec:nu}
We show that the number of primary factors with the full and Hodge-class ranks of the cubic
zero factor and the exact residue polynomial $\det(tI-R_x^\sharp)$ on the lattice of
Section~\ref{sec:quantum} can be tracked under birational morphism of smooth projective complex varieties. The number is denoted by $n_{\mathrm{cub}}$, rather than
Gu\'er\'e's $\nu^X_{\ev,\alpha}$, to emphasis that the full lattice is also part of our test.

\subsection{Counted components}
Assume first that $X$ is connected, and put $n=\dim_\QQ H^*(X,\QQ)$.
Consider the trace matrix $\cT(\kappa_\tau)=(\tr\kappa_\tau^{i+j})_{0\le i,j<n}$, where $\kappa_\tau^{i+j}$ denote the $i+j$-th self composition of $\kappa_\tau$, with $\kappa_{\tau}^{0}=I$. At a rigid point $x$ with distinct eigenvalues
$\alpha_1,\ldots,\alpha_k$ of multiplicities $m_a, 1\leq a\leq k$, let
$V=(\alpha_a^i)_{0\le i<n,\,1\le a\le k}$. Then
\[
 \cT(\kappa_\tau)(x)=V\operatorname{diag}(m_a)V^{\mathsf T},\qquad
 \det\bigl(\tr\kappa_\tau(x)^{i+j}\bigr)_{0\le i,j<k}
 =\Bigl(\prod_a m_a\Bigr)\prod_{a<b}(\alpha_b-\alpha_a)^2\ne0.
\]
Consequently $\rank\cT(\kappa_\tau)(x)$ is the number of distinct eigenvalues. Let
\[
 r=\max_{x\text{ rigid}}(\rank\cT(\kappa_\tau)(x)),\qquad
 U_X=\{x\in B_X:\rank\cT(\kappa_\tau)(x)=r\}.
\]
The complement of $U_X$ is a proper analytic subset, defined by the $r$-minors.
Since $B_X$ is smooth and connected, it is irreducible, and $U_X$ is
connected by \cite[Corollary~2.7]{Hansen}. By the identity principle
\cite[Lemma~2.2.3]{Conrad}: if on a nonempty admissible open subset (e.g. a suffciently small analytic open disc), we have $\rank\cT(\kappa_\tau)(x)\equiv k$, then $r=k.$ In particular, the maximum number of eigenvalues is unchanged by nonempty open restriction. $U_X$ is called the maximal locus of $B_X$.

As in \cite{KKPY}, we consider the \emph{reduced spectral cover} over the maximal locus $\widetilde U_X\to U_X$:
\[
 \widetilde B_X=
 \{(x,\alpha)\in B_X\times(\bA^1_\KK)^{\an}:
                \det(\kappa_\tau(x)-\alpha I)=0\}_{\red},\qquad
 \widetilde U_X=\widetilde B_X\times_{B_X}U_X.
\]
It is finite because the characteristic polynomial is monic. We show that  $\widetilde U_X\to U_X$ is actually finite \'etale of degree $r$.
Near a rigid point $x_0\in U_X$, factor the characteristic polynomial
analytically as
\[
 \chi(T)=\det(TI-\kappa_\tau)=\prod_{a=1}^r\chi_a(T),
 \qquad \chi_a(T)|_{x_0}=(T-\alpha_a(x_0))^{m_a}.
\]
Such a decomposition exists, after shrinking $U_X$ to an analytic open $W$, follows from the factorization form of Hensel's lemma for rigid analytic local rings. It is clear that for each $x\in W$, $\chi_a(T)(x)$ has exactly one root that can be written as $\alpha_a=\frac1{m_a}\tr(\kappa_\tau|_{E_a})$.
Since the base is reduced, the resulting
fiberwise identities give identities of analytic coefficients:
\[
 \det(TI-\kappa_\tau|_{E_a})=(T-\alpha_a)^{m_a},
 \qquad \alpha_a-\alpha_b\in\cO^\times\quad(a\ne b).
\]
Consequently the reduced cover locally has the description
\[
 \widetilde U_X\simeq\coprod_{a=1}^r\{(x,\alpha_a(x))\}.
\]
Hence $\widetilde U_X\to U_X$ is actually finite \'etale of degree $r$. See also \cite[Section~5.2.2]{KKPY}.

Every connected component $\Sigma$ of $\widetilde{U}_X$ is smooth and irreducible and
surjects onto $U_X$, since its image is nonempty, open and closed under a finite \'etale morphism. In particular, there are only
finitely many components, and $\sum_{\Sigma}\deg(\pi_\Sigma)=r$. Write $\pi_\Sigma:\Sigma\to U_X$.

By the uniqueness in Lemma~\ref{lem:separation}, the primary horizontal
projectors glue on $\Sigma$ to a formal projector $P_\Sigma(u)$. Its
coefficients are analytic, and it is convergent locally in the base and
near $u=0$. Let
\[
 \mathscr L_\Sigma=P_\Sigma(u)\pi_\Sigma^*\cL_X,
 \qquad
 L_\Sigma=P_\Sigma(u)\pi_\Sigma^*\cL_X^{\mathrm{Hdg}},
\]
\[
 r_\Sigma=\rank\mathscr L_\Sigma,
 \qquad \rho_\Sigma=\rank L_\Sigma.
\]
For a rigid $(x,\alpha)\in\Sigma$, these ranks are $r_\Sigma=\dim_\KK E^X_{x,\alpha}, \rho_\Sigma=\rho^X_{x,\alpha}$ defined in \cite{Guere}. By the above discussions, these ranks are locally constant.

Suppose that $\rho_\Sigma=2$, and write $L=L_\Sigma$. The selected leading
endomorphism has a unique eigenvalue, equal to its half-trace. Thus
\[
 \alpha=\tfrac12\tr(\kappa_{\tau,L}),\qquad
 N=\kappa_{\tau,L}-\alpha I,\qquad N^2=0.
\]
Define $\Sigma^\circ=\{x\in\Sigma:N_x\ne0\}.$ If $\Sigma^\circ$ is nonempty, it is the complement of a proper closed analytic subset
of the smooth irreducible space $\Sigma$. So it is connected by
\cite[Corollary~2.7]{Hansen}, and hence is irreducible. On $\Sigma^\circ$, fix the notations as in Section \ref{sec:quantum}:
\[
 \rank N=1,\qquad \ell=\ker N=\im N,
 \qquad L^\sharp=\{v\in L:v\bmod u\in\ell\},
 \qquad \nabla^0_{\partial_u}=\nabla_{\partial_u}+\alpha/u^2.
\]
By Lemma~\ref{lem:residue-lattice}, the obstruction
\[
 (u^2\nabla^0_{\partial_u})\bmod u
 \in\End_{\cO_{\Sigma^\circ}}(L^\sharp/uL^\sharp)
\]
is analytic and vanishes precisely when $L^\sharp$ is logarithmic. Suppose
that it vanishes at one rigid point $x_0\in\Sigma^\circ$.
Proposition~\ref{prop:transport} and Lemma~\ref{lem:residue-lattice} imply
that it vanishes at every rigid point of a neighborhood of $x_0$. On that
reduced neighborhood it therefore vanishes as an analytic section. The
identity principle now gives
\[
 (u^2\nabla^0_{\partial_u})\bmod u=0
 \quad\text{on all of }\Sigma^\circ.
\]
Hence The residue $R^\sharp=(u\nabla^0_{\partial_u})\bmod u$ is an analytic endomorphism on $\Sigma^\circ$. In particular,
\[
 \det(tI-R_x^\sharp)=\det(tI-R_{x_0}^\sharp)
 \qquad(x\in\Sigma^\circ\text{ rigid}).
\]
Finally, define 
$$P_{\mathrm{cub}}(t)=(t-1/6)(t-5/6)=t^2-t+5/36.$$

\begin{definition}\label{def:nu}
A connected component $\Sigma$ of $\widetilde U_X$ is \emph{counted} if
\[
 r_\Sigma=12,\qquad \rho_\Sigma=2,\qquad
 \Sigma^\circ\ne\varnothing,
\]
and, on $\Sigma^\circ$,
\[
 u\nabla^0_{\partial_u}(L^\sharp)\subseteq L^\sharp,
 \qquad \det(tI-R^\sharp)=P_{\mathrm{cub}}(t).
\]
The last two conditions may be checked at one rigid point
of $\Sigma^\circ$. Define
\[
 n_{\mathrm{cub}}(X)
   =\sum_{\Sigma\ \mathrm{counted}}\deg(\Sigma\to U_X).
\]
For a disconnected variety $X$, we define $n_{\mathrm{cub}}(X)$ by simply sum over its connected components. For an
external direct sum on a product of connected bases, use the same
construction for the direct-sum connection and the direct-sum Hodge-class
subbundle.
\end{definition}

\begin{proposition}\label{prop:invariance}
The count $n_{\mathrm{cub}}$ is unchanged by restriction to a nonempty
connected admissible open, by pullback to a nonempty connected finite
\'etale cover of the base, and by an isomorphism of the full connections
identifying their Hodge-class subbundles and full lattices. It is additive for
external direct sums on products of bases.
\end{proposition}
\begin{proof}
\emph{Restriction.}
Let $V\subset B_X$ be a nonempty connected admissible open. We have
\[
 r_V=r,\qquad U_V=V\cap U_X,
 \qquad \widetilde U_V=\widetilde U_X\times_{U_X}U_V.
\]
The space $U_V$ is nonempty and connected by the same argument as above. If
\[
 \Sigma\times_{U_X}U_V=\coprod_b\Sigma_b,
\]
then each $\Sigma_b\to U_V$ is finite \'etale and surjective, and
\[
 \sum_b\deg(\Sigma_b\to U_V)=\deg(\Sigma\to U_X).
\]
Each $\Sigma_b$ is a nonempty admissible open in $\Sigma$. Thus, if
$\Sigma^\circ\ne\varnothing$, the identity principle implies
$\Sigma_b\cap\Sigma^\circ\ne\varnothing$. The ranks are unchanged, and
logarithmicity and the polynomial condition on $\Sigma_b$ hold if and
only if they hold on $\Sigma$, because it holds at one rigid point.
If $N$ is identically zero on $\Sigma$, it remains so after restriction.
Therefore all the components $\Sigma_b$ are counted precisely when
$\Sigma$ is counted. The degree identity proves restriction invariance.

\smallskip
\emph{Finite \'etale pullback.}
Let $g:B'\to B_X$ be nonempty, connected and finite \'etale. Its image
is open and closed, so it is surjective. The pulled-back leading matrix
has the same spectrum on corresponding geometric fibers. Consequently
\[
 U'=g^{-1}(U_X),\qquad
 \widetilde U'\simeq\widetilde U_X\times_{U_X}U'.
\]
Again by \cite[Corollary~2.7]{Hansen}, $U'$ is connected. For each $\Sigma$, decompose
\[
 \Sigma\times_{U_X}U'=\coprod_b\Sigma'_b.
\]
The projections from $\Sigma'_b$ to both $\Sigma$ and $U'$ are finite
\'etale with nonempty open-and-closed images, hence are surjective.
The ranks and the condition $N\ne0$ are preserved and reflected under
this pullback. The same is true of logarithmicity and of the residue
polynomial: the obstruction and residue pull back as endomorphisms, and
one may test them at corresponding rigid points.

\smallskip
\emph{Connection isomorphisms.}
Suppose an isomorphism of bases $f$ and a regular invertible matrix
$G(x,u)$ identify the connections over $U=cu$, where $c\in\KK^\times$
is constant. Comparing the leading vertical coefficients gives
\[
 G(x,0)^{-1}c^{-1}\kappa'_{\tau}(f(x))G(x,0)=\kappa_\tau(x).
\]
It therefore identifies the reduced spectral covers by
$(x,\alpha)\mapsto(f(x),c\alpha)$ and preserves the full and Hodge-class
ranks. On a rank-two Hodge-class factor,
\[
 \alpha'(f(x))=c\alpha(x),\qquad
 G(x,0)^{-1}c^{-1}N'(f(x))G(x,0)=N(x).
\]
Thus the original isomorphism also commutes with the normalized vertical
connections. It preserves the full lattices, carries $\ker N$ to
$\ker N'$, and consequently identifies $L^\sharp$ with the corresponding
pullback of $(L')^\sharp$. Lemma~\ref{lem:residue-lattice} preserves the vanishing of the
obstructions and, when they vanish, conjugates the residues. The counted
components and their covering degrees are therefore preserved.

\smallskip
\emph{External direct sums.}
Consider finitely many connected varieties $X_j$, their bases
$B_j=B_{X_j}$, and the external direct sum over $\prod_j B_j$. It suffices to work over a
nonempty product of sufficiently small connected neighborhoods inside
the individual maximal spectral loci.

Choose a rigid point in each maximal spectral locus. Each summand has an
independent unit parameter. By the string equation \cite[Section~1.3]{Guere}, translating $\tau$ by $a_j$ replaces
its leading operator by
\[
 \kappa_j\longmapsto\kappa_j+a_jI.
\]
Choose the $a_j$ so that the finitely many sets of eigenvalues are pairwise
disjoint, this can be done because at each step only finitely many scalar values are forbidden.
 After shrinking around the resulting
points, the spectral covers split into analytic sheets. The number of eigenvalues of the external
sum is now $\sum_j r_j$. This number is clearly maximum as a direct sum has at most $\sum_j r_j$ distinct
eigenvalues at any point.

On this product neighborhood, every primary projector of the external sum
is the pullback of exactly one primary projector of a summand. Its full
and Hodge-class ranks, $N$, lattice and its obstruction
are the corresponding pullbacks. The condition $N\not\equiv0$ is
preserved and reflected under a product projection. If the pulled-back
factor is logarithmic with the required polynomial at one rigid point,
so is the original factor at its projected point, and conversely. Hence 
\[
 n_{\mathrm{cub}}(\text{external direct sum})
   =\sum_j n_{\mathrm{cub}}(X_j).
\]
\end{proof}
To summarize, we have the following corollary which allows us to calulate $ n_{\mathrm{cub}}(X)$ locally:
\begin{corollary}\label{Cor:local-count}
    Let $x\in B_X$ be a rigid point. Suppose every primary factor at $x$ has
either Hodge-class rank one, or Hodge-class rank two with nonscalar
leading endomorphism. Then $x\in U_X$. Index these factors by $a=1,\ldots,k$, with full rank $r_a$, Hodge-class
rank $\rho_a$, and Hodge-class lattice $L_a$. Then
\[
 n_{\mathrm{cub}}(X)
 =\#\left\{a:
 \begin{array}{l}
 r_a=12,\ \rho_a=2,\quad L_a^\sharp\text{ is logarithmic},\\
 \det(tI-R_a^\sharp)=P_{\mathrm{cub}}(t)
 \end{array}\right\}.
\]
\end{corollary}

\subsection{The local blowup comparison}\label{sec:comparison}
In this subsection, we provide a blow-up formula for our counting. We show that Iritani's full quantum $D$-module decomposition over the non-Archimedean field $\KK$ applies to our set up in Section~\ref{sec:quantum}. Then main ideas of the proof are the same as in \cite{iritani2025quantumcohomologyblowups,iritani2026notesdecompositiontheoremblowups} and \cite[Theorem~4.5]{KKPY}, while we provide details showing that one indeed get a rigid point in $B_X\times \DD_u$ after taking evaluations.
\begin{lemma}\label{lem:homogeneous}
Give $v$ a positive rational degree and $x_1,\ldots,x_N$ fixed rational
degrees. A homogeneous series
\[
 F(x,v)=\sum_{\boldsymbol n\in\NN^N,\,\ell\in\ZZ}
 a_{\boldsymbol n,\ell}x^{\boldsymbol n}v^\ell,
 \qquad a_{\boldsymbol n,\ell}\in\CC,
\]
converges normally on every fixed closed annulus $0<a\le |v|\le b$ for
sufficiently small $|x_i|$. Common radii work for finitely many such series
and their first derivatives, including when $u$ is among the $x_i$.
\end{lemma}
\begin{proof}
Homogeneity determines at most one $\ell$ for each $\boldsymbol n$ and gives
$|\ell|\le C_0+C_1|\boldsymbol n|$, where $|\boldsymbol n|=\sum_i n_i$.
For $\Lambda=\max(1,b,a^{-1})$ and $|x_i|\le\varepsilon$,
\[
 |a_{\boldsymbol n,\ell}x^{\boldsymbol n}v^\ell|
 \le\Lambda^{C_0}(\varepsilon\Lambda^{C_1})^{|\boldsymbol n|}.
\]
Choose $\varepsilon\Lambda^{C_1}<1$. Since the norm on $\CC$ is trivial and
there are finitely many multi-indices of bounded total degree, this gives
normal convergence. Derivatives satisfy the same support bound, with only
fixed radius factors. This is the homogeneous estimate of
\cite[Lemma~4.6]{KKPY}.
\end{proof}
Now we provide the key technical lemma which allows us to rescale the evaluated target parameters into
the chosen analytic bases.
\begin{lemma}\label{lem:homothety}
For $\lambda\in\KK^\times$, define
$$ S_\lambda(Q^\beta)=\lambda^{2c_1(X)\cdot\beta}Q^\beta,
 \qquad S_\lambda(T_i)=\lambda^{2-d_i}T_i,\qquad
 D_\lambda|_{H^d(X,\KK)}=\lambda^dI.$$
On any domain where both sides converge, the map
$F_\lambda(x,u)=(S_\lambda x,\lambda^2u)$ satisfies
\[
 F_\lambda^*\nabla=D_\lambda^{-1}\nabla D_\lambda.
\]
Thus $D_\lambda$ identifies the pulled-back connection with the original
one. It preserves the full and Hodge-class lattices and induces a
comparison of their normalized vertical connections. 

Furthermore, on rank-two primary
factors with $N\ne0$, it preserves logarithmicity of $L^\sharp$ and its
exact residue characteristic polynomial.
\end{lemma}
\begin{proof}
Writing $e_j=\deg\phi_j$, the coefficient of $\phi_k$ in
$\kappa_\tau\phi_j$ has degree $2+e_j-e_k$. Similarly, the coefficient
in $(\alpha_i\star_\tau)\phi_j$ has degree $d_i+e_j-e_k$. Hence
\[
 \kappa_\tau(S_\lambda x)=\lambda^2D_\lambda^{-1}\kappa_\tau(x)D_\lambda,
 \qquad
 (\alpha_i\star_{S_\lambda x})=
 \lambda^{d_i}D_\lambda^{-1}(\alpha_i\star_x)D_\lambda.
\]
Pullback by $U=\lambda^2u$ gives
\[
 (F_\lambda^*\nabla)_{\partial_u}
 =\partial_u-\frac{\kappa_\tau(S_\lambda x)}{\lambda^2u^2}
       +\frac{\Gr_X}{u}
 =D_\lambda^{-1}\nabla_{\partial_u}D_\lambda,
\]
since $D_\lambda$ commutes with $\Gr_X$. In each $T_i$-direction,
$\lambda^{2-d_i}\lambda^{d_i}\lambda^{-2}=1$, the divisor logarithmic
directions satisfy the same identity. See also
\cite[Section~2.3]{iritani2025quantumcohomologyblowups}.

The constant frame map preserves Hodge classes and the lattices. On
corresponding rank-two factors,
\[
 \alpha(S_\lambda x)=\lambda^2\alpha(x),\qquad
 N(S_\lambda x)=\lambda^2D_\lambda^{-1}N(x)D_\lambda,\qquad
 \lambda^2\frac{\alpha(S_\lambda x)}{(\lambda^2u)^2}
 =\frac{\alpha(x)}{u^2}.
\]
Thus normalization is compatible, and the comparison carries $\ker N$
and $L^\sharp$ to their counterparts. 
Lemma~\ref{lem:residue-lattice} then proves our claim. 
\end{proof}

\begin{theorem}\label{thm:local-comparison}
Let $X$ be a connected smooth projective complex variety, let
$X'\subset X$ be a smooth connected center of codimension $r\ge2$, and let 
$$\pr:\widetilde X=\Bl_{X'}X\to X$$ 
be the blowup of $X$ along $X'$. After passing to
finite \'etale covers of the Novikov tori $T_X$ if necessary, there are nonempty connected admissible open neighborhoods $V\subset B_{\widetilde X}$ and
$V'\subset B_X\times B_{X'}^{r-1}$ and an isomorphism $f:V\xrightarrow{\sim}V'$.
For a constant $\lambda\in\KK^\times$, the map
\[
 \Phi(x,u)=(f(x),\lambda^2u)
\]
lifts to an isomorphism of the full quantum connections identifying their
Hodge-class subbundles. The isomorphism and its inverse are analytic and
regular at $u=0$. In particular, their $u$-adic completions identify the full and Hodge-class lattices.
\end{theorem}
\begin{proof}
Let $E$ be the exceptional divisor, $i:X'\hookrightarrow X$ the inclusion,
and $N_i$ its normal bundle. As in the blowup notation of
\cite[Section~2.4]{Guere}, put
\[
 \mathfrak{s}=
 \begin{cases}
 r-1,&r\text{ even},\\
 2(r-1),&r\text{ odd},
 \end{cases}
 \qquad q'=q^{1/\mathfrak{s}},\qquad \deg q=2(r-1).
\]
Then $\deg q'\in\{1,2\}$. 

Choose graded cohomological bases and Hodge-class bases compatible with
the blowup decomposition. By
\cite[Theorem~5.18(1), (3), (5)--(7)]{iritani2025quantumcohomologyblowups},
there are a parameter map and a bundle isomorphism
\[
 \begin{gathered}
 \sfF_{Q,q}(\widetilde T)
   =(\tau_0(Q,q,\widetilde T),\ldots,\tau_{r-1}(Q,q,\widetilde T)),\\
 \widetilde\Psi(u):
 \tau_0^*\operatorname{QDM}(X)^{\mathrm{la}}
 \oplus\bigoplus_{j=1}^{r-1}\tau_j^*\operatorname{QDM}(X')^{\mathrm{la}}
 \xrightarrow{\sim}\operatorname{QDM}(\widetilde X)^{\mathrm{la}}.
 \end{gathered}
\]
Here $\widetilde T$ are the scalar coordinates of $\widetilde\tau$, 
$\tau_0$ is $H^*(X)$-valued and $\tau_j$ is $H^*(X')$-valued for $j\ge1$.
Our $\widetilde\Psi$ is the inverse of Iritani's map. Both bundle matrices
are defined over the graded ring $ \CC[u]((q^{-1/\mathfrak s}))[[Q,\widetilde T]],$ so they have no negative powers of $u$. The map commutes with the full connection
operators displayed in (5.41)--(5.43) of \cite{iritani2025quantumcohomologyblowups}. Their vertical
grading operators are exactly $ \Gr_Y|_{H^d(Y)}=(d-\dim Y)/2, Y=\widetilde X,X,X'$ as in Section~\ref{sec:quantum}.

Each $\tau_j$ has total degree 2, so a scalar
parameter coordinate corresponding to degree $d$ has degree $2-d$.
The matrix entries of $\widetilde\Psi$ and its inverse are homogeneous
when the fixed grading shifts on the center summands are included.
These shifts then do not introduce negative powers
of $u$ into either matrix. The parameter Jacobian at $Q=\widetilde T=0$
is invertible. By \cite[Proposition~8, Remark~9 and Corollary~11]
{iritani2026notesdecompositiontheoremblowups}, these maps are Hodge-equivariant.

Choose integral ample classes $H_1,\ldots,H_{\rho(X)}$ forming a rational
basis of $N^1(X)$, and put $\omega=\sum_iH_i$. Up to a finite \'etale cover of the Novikov tori, we write
\[
 Q^\beta=\prod_i Q_i^{H_i\cdot\beta},
 \qquad c_1(X)=\sum_i c_iH_i,\qquad \deg Q_i=2c_i.
\]
For effective $\beta$, all exponents $H_i\cdot\beta$ are nonnegative.
To compare data in $\widetilde X$ and $X'$ with data $X$, we set the common Novikov substitutions:
$$ \widetilde Q^{\widetilde\beta}
   =Q^{\pr_*\widetilde\beta}q^{-[E]\cdot\widetilde\beta},
 \qquad
 (Q')^{\beta'}
   =Q^{i_*\beta'}q^{-c_1(N_i)\cdot\beta'/(r-1)}.$$
These are the substitutions in
\cite[(1.1), (5.38)--(5.40)]{iritani2025quantumcohomologyblowups}.
Degrees are preseved under this substitutions as we have
\[
 c_1(\widetilde X)=\pr^*c_1(X)-(r-1)[E],
 \qquad c_1(X')=i^*c_1(X)-c_1(N_i).
\]
 Fix a root $q'_*=s^{1/\mathfrak s}$ and a closed annulus
$0<a\le|q'|\le b$ containing it, with $a,b\in|\KK^\times|$. By
\cite[Remark~1.3 and Section~2.2]{iritani2025quantumcohomologyblowups}, the
completions above are graded. Apply Lemma~\ref{lem:homogeneous} with
$v=q'$ and variables $Q_i,\widetilde T_i,u$ to the entries of
\[
 \sfF_{Q,q}(\widetilde T),\qquad \widetilde\Psi(Q,q,\widetilde T,u),
 \qquad \widetilde\Psi(Q,q,\widetilde T,u)^{-1},
\]
using $v=q'$ and $x_i=Q_i,\widetilde T_i,u$, omitting $u$ for the
parameter map. These entries and their first
derivatives converge normally on common sufficiently small closed discs.
The formal inverse identities hold in the resulting Banach algebra, so
both bundle maps are analytic and regular at $u=0$.

Now we evaluate the above formal constructions to a rigid point. Fix $q'_*=s^{1/\mathfrak{s}}$ and choose the range of $q'$ above to
contain it. Let $q_*=s, Q_{i,*}=s^R, \lambda=s^{-1}.$
The annulus, the parameter discs and $\lambda$ are fixed independently of
$R$. We choose $R\gg0$ after imposing the finitely many conditions below.
These values define
\[
 Q_*^\beta=s^{R\omega\cdot\beta},\qquad
 \widetilde Q_*^{\widetilde\beta}
 =s^{(R\pr^*\omega-E)\cdot\widetilde\beta},\qquad
 v_{\widetilde X}(\widetilde Q_*)=R\pr^*\omega-[E].
\]
Since $-E$ is relatively ample, this class is ample for all sufficiently
large $R$. As a result, together with $\widetilde T=0$, it gives a point of $B_{\widetilde X}$.
Since all Hodge-class parameters $\widetilde T$, including the divisor
parameters, will vary in a small neighborhood, we indeed get an open neighborhood inside $B_{\widetilde X}$. Write $\sfF_R(\widetilde\tau)=\sfF_{Q_*,q_*}(\widetilde \tau)$ for the specialized map.

We next place all target summands in $B_X\times B_{X'}^{r-1}$.
At the auxiliary locus $Q=\widetilde T=0$, a nonzero scalar target
coordinate multiplying a degree-$d$ class is a complex multiple of
$q^{(2-d)/(2(r-1))}$ by homogeneity. After $S_\lambda$ its valuation is
$$
 v\left(\lambda^{2-d}q_*^{(2-d)/(2(r-1))}\right)
   =(d-2)\left(1-\frac1{2(r-1)}\right)>0
 \qquad(d\ge4).$$
Thus all nonzero initial higher-degree coordinates lie strictly inside
the required open unit discs.

Write $Y_0=X$, $Y_j=X'$ for $j\ge1$, and let $\delta_j$ be the divisor
part of $\tau_j$. The initial-value formulas
\cite[Theorem~5.18(6), (5.19)]{iritani2025quantumcohomologyblowups} give
\[
 \delta_j(0,q,0)=h_j,\qquad h_0=0,\qquad
 h_j=\frac{2\pi\mathrm i(j-1/2)}{r-1}c_1(N_i)\quad(1\le j\le r-1).
\]
Thus $|\delta_j-h_j|<1$ on a common smaller neighborhood for all large
$R$, where the norm is the maximum coefficient norm in an integral
rational divisor basis. In particular,
$|(\delta_j-h_j)\cdot\beta|<1$ for every integral curve class $\beta$.
Let $Q_{Y_0,*}=Q_*$ and, for $j\ge1$, let $Q_{Y_j,*}^\beta=Q_*^{i_*\beta}q_*^{-c_1(N_i)\cdot\beta/(r-1)}.$
Distinguish the divisor-absorbed, unscaled characters from the final
scaled characters:
\[
 \begin{aligned}
 Q_{Y_j,\mathrm{abs}}^\beta
   &:=Q_{Y_j,*}^\beta\exp_\CC(h_j\cdot\beta)
                 \exp_\KK((\delta_j-h_j)\cdot\beta),\\
 Q_{Y_j}^\beta&:=\lambda^{2c_1(Y_j)\cdot\beta}
                     Q_{Y_j,\mathrm{abs}}^\beta.
 \end{aligned}
\]
Since
$v(\lambda)=-1$, the scaled valuation classes are
\[
 R\omega-2c_1(X),\qquad
 Ri^*\omega-\frac{c_1(N_i)}{r-1}-2c_1(X').
\]
They are ample for $R\gg0$. Together with the higher-coordinate bounds,
this defines the target chart
\[
 \sfA_{\mathrm{tar}}((\tau_j)_j)
 =\left( Q_{Y_j},\ \lambda^2\tau_j^{(0)},\
           (\lambda^{2-d}\tau_j^{(d)})_{d\ge4}\right)_{0\le j\le r-1}
\]
on a neighborhood of $\sfF_R(0)$, with image in the prescribed target base $B_X\times B_{X'}^{r-1}$.

Now we are ready to calculate quantum connections. 
The unscaled target parameters may be large, so convergence of the maps
alone is insufficient. For $Y=X$ or $X'$, put $\omega_0=\omega$ or
$i^*\omega$. After the string and divisor equations, a term
$Q_Y^\beta\prod_{d_l\ge4}T_l^{n_l}$ in the coefficient of $\phi_k$ in
$\phi_i\star_\tau\phi_j$ satisfies
\begin{equation}\label{eq:dimension}
 \sum_{d_l\ge4}(d_l-2)n_l
 =2c_1(Y)\cdot\beta+\deg\phi_k-\deg\phi_i-\deg\phi_j.
\end{equation}
For a fixed divisor $D$, choose $C_D$ with $C_D\omega_0\pm D$ ample;
then $|D\cdot\beta|\le C_D\omega_0\cdot\beta$ for effective $\beta$.
Thus $\sum_l n_l\le A+B\omega_0\cdot\beta$ for fixed $A,B\ge0$.
Let $M\ge1$ bound the unscaled higher-degree coordinates on the common
closed neighborhood. Their derivatives are bounded there as well.
For a constant $C$ independent of $R$, we have
\[
 |Q_Y^\beta|\le|s|^{(R-C)\omega_0\cdot\beta},\qquad
 |\mathrm{term}|\le
 M^A\bigl(M^B|s|^{R-C}\bigr)^{\omega_0\cdot\beta}.
\]
Choose $R$ with $M^B|s|^{R-C}<1$. By Lemma~\ref{lem:finite} and
\eqref{eq:dimension}, bounded ample degree leaves finitely many terms.
The target products and their composed first derivatives therefore
converge normally. 

For upstairs, fix $R_0$ with $A_0=R_0\pr^*\omega-E$ ample.
On $|Q_i|\le|s|^R$, $|q|=|s|$, $R\ge R_0$, one has
\[
 \bigl|Q^{\pr_*\widetilde\beta}q^{-E\cdot\widetilde\beta}\bigr|
 \le |s|^{A_0\cdot\widetilde\beta}.
\]
The estimates hold before specialization, hence the formal
identities as after (5.36) in \cite{iritani2025quantumcohomologyblowups},
remain analytic identities. Apply Lemma~\ref{lem:homothety} gives the corresponding result in 
scaled case.

Let $J_0(q)$ be the Hodge-class Jacobian of $\sfF$ at
$Q=\widetilde T=0$. It is invertible by the formal theorem and Hodge-equivariance. Moreover,
\[
 \dim B_{\widetilde X}
 =\dim_\QQ H^*(X)^{\hdg}+(r-1)\dim_\QQ H^*(X')^{\hdg}
 =\dim(B_X\times B_{X'}^{r-1}).
\]
If source and target scalar coordinate weights are $w_a,w'_b$, the
$(b,a)$ entry of $J_0(q)$ has weight $w'_b-w_a$. At $Q=\widetilde T=0$
it is zero or a Laurent monomial in $q'$; all determinant terms have
weight $\sum_b w'_b-\sum_a w_a$. Consequently
\[
 \det J_0(q)=c(q')^m,\quad c\in\CC^\times,\qquad
 J_R:=d_{\widetilde T}\sfF_R(0)\longrightarrow J_0(q_*).
\]
For $R\gg0$,
\[
 J_R=J_0(q_*)\bigl(I+J_0(q_*)^{-1}(J_R-J_0(q_*))\bigr),\qquad
 \|J_0(q_*)^{-1}(J_R-J_0(q_*))\|<1.
\]
Thus $J_R$ is invertible, and the analytic inverse-function theorem gives
an isomorphism from a germ at $\widetilde T=0$ to one at $\sfF_R(0)$.

Let $\widetilde\delta$ be the source divisor part. The source chart is
\[
 \sfA_{\mathrm{src}}(\widetilde T)
 =\left(\widetilde Q_*e^{\widetilde\delta},\
        \widetilde\tau^{(0)},(\widetilde\tau^{(d)})_{d\ge4}\right),\qquad
 (\widetilde Q_*e^{\widetilde\delta})^{\widetilde\beta}
 :=\widetilde Q_*^{\widetilde\beta}
                \exp_\KK(\widetilde\delta\cdot\widetilde\beta).
\]
For either this chart or $\sfA_{\mathrm{tar}}$, a divisor basis $D_a$ and
numerical curve basis $\gamma_b$ give
\[
 d\log Q^{\gamma_b}=\sum_a(D_a\cdot\gamma_b)\,d\delta_a,
 \qquad\det(D_a\cdot\gamma_b)\ne0.
\]
Here $d\log z=z^{-1}dz$; no logarithm of a non-small constant is evaluated.
The remaining Jacobian blocks are identities or nonzero scalings. Thus
both charts are local isomorphisms by the argument of
Lemma~\ref{lem:divisor}, with target membership checked above. After
shrinking, we have the claimed isomorphism:
$$f=\sfA_{\mathrm{tar}}\circ\sfF_R\circ\sfA_{\mathrm{src}}^{-1}:
 V\xrightarrow{\sim}V'.$$

Let $\nabla_{\mathrm{tar}}$ be the external sum of the connection of $X$
and $r-1$ copies of that of $X'$, and let $\cD_\lambda$ be the corresponding
direct sum of $D_\lambda$. Define
\[
 \begin{aligned}
 \Phi(x,u)&=(f(x),\lambda^2u),\\
 G(x,u)&=\cD_\lambda^{-1}
 \widetilde\Psi\bigl(Q_*,q_*,\sfA_{\mathrm{src}}^{-1}(x),u\bigr)^{-1}.
 \end{aligned}
\]
The first factor $\widetilde\Psi^{-1}$ maps the source to the unscaled
target; $\cD_\lambda^{-1}$ then maps it to the scaled pullback. For
$\nabla_{\mathrm{src}}=d+\Omega_{\mathrm{src}}$ and
$\nabla_{\mathrm{tar}}=d+\Omega_{\mathrm{tar}}$, the evaluated identities give
\[
 \begin{gathered}
 dG(x,u)+\Phi^*\Omega_{\mathrm{tar}}G(x,u)
       =G(x,u)\Omega_{\mathrm{src}},\\
 G(x,u)^{-1}\Phi^*\nabla_{\mathrm{tar}}G(x,u)=\nabla_{\mathrm{src}}.
 \end{gathered}
\]
Both matrices are analytic, regular at $u=0$, and preserve the Hodge-class
subbundles. Shrink the loop disc so that $\lambda^2u$ lies in the target
disc. Their completions identify the original full and Hodge-class
lattices over $\cO_V[[u]]$.
On corresponding primary factors, comparison of the leading terms gives
\[
 G(x,0)^{-1}\lambda^{-2}\kappa_{\mathrm{tar}}(f(x))G(x,0)
   =\kappa_{\mathrm{src}}(x).
\]
On their rank-two Hodge-class parts this gives
\[
 \alpha_{\mathrm{tar}}(f(x))=\lambda^2\alpha_{\mathrm{src}}(x),\qquad
 G(x,0)^{-1}\lambda^{-2}N_{\mathrm{tar}}(f(x))G(x,0)=N_{\mathrm{src}}(x).
\]
Thus the normalized connections and $L^\sharp$ are identified. Since $u\partial_u=U\partial_U$,
Lemma~\ref{lem:residue-lattice} preserves logarithmicity and the exact
residue polynomial. Choosing one $R$ satisfying all preceding conditions
and shrinking to connected neighborhoods proves the theorem.

\end{proof}

\begin{corollary}\label{cor:blowup}
For every smooth center $X'\subset X$ of pure codimension $r\ge2$,
\[
 n_{\mathrm{cub}}(\Bl_{X'}X)
   =n_{\mathrm{cub}}(X)+(r-1)n_{\mathrm{cub}}(X').
\]
\end{corollary}
\begin{proof}
First suppose the ambient variety and center are connected.
Theorem~\ref{thm:local-comparison} identifies the blowup connection with
an external direct sum of one copy for $X$ and $r-1$ copies for $X'$,
after finite \'etale torus covers and nonempty open restriction.
All these operations preserve the count by
Proposition~\ref{prop:invariance}, and external direct sums are additive.
Thus
\[
 n_{\mathrm{cub}}(\Bl_{X'}X)
  =n_{\mathrm{cub}}(X)+\sum_{j=1}^{r-1}n_{\mathrm{cub}}(X')
  =n_{\mathrm{cub}}(X)+(r-1)n_{\mathrm{cub}}(X').
\]
For a disconnected smooth center, its connected components are disjoint.
Blow them up successively. Summing the connected formulas gives the result. Finally apply the same
argument on the connected components of a disconnected $X$.
\end{proof}

\section{Proof of the main theorem}\label{sec:main-proof}
We apply Cai's cubic--curve obstruction and then exclude smooth projective surfaces. The transport and blowup results of the preceding sections provide the compatibility needed for weak factorization.

Cai works with even cohomology \cite[Section~1]{Cai}, which for a smooth cubic threefold and for a smooth projective curve equals the complexified Hodge-class subspace. His Propositions~6 and~7 distinguish the relevant factors by their exponents modulo $\ZZ$. We use the following lattice form of his calculation, which records the exact residue polynomial on the specified lattice, without repeating his matrix reduction.

\begin{lemma}\label{lem:cubic}
Let $Y\subset\PP^4_\CC$ be a smooth cubic threefold. At $T_i=0$ and a nonzero sufficiently small line parameter $Q_*$,
\[
 \mathrm{Sp}^Y_{\ev}=\{0,6\sqrt{3Q_*},-6\sqrt{3Q_*}\},\qquad
 \dim_\KK E^Y_{\ev,0}=12,\quad \rho^Y_{\ev,0}=2,\quad
 \gamma^Y_{\ev,0}=1.
\]
The two nonzero primary factors have rank one. On the Hodge-class part of the zero factor, $L^\sharp$ is logarithmic with residue polynomial $P_{\mathrm{cub}}$.
\end{lemma}

\begin{proof}
The even Jordan decomposition and fractional exponents are computed in \cite[Section~3 and Proposition~6]{Cai}; compare the small-quantum spectrum in \cite[Example~6.21]{KKPY}. The only odd cohomology is $H^3(Y,\QQ)$, of dimension 10 \cite{CG}.

Quantum multiplication by the hyperplane class vanishes there. In fact, the dimension equation for $\langle H,\gamma_1,\gamma_2\rangle_{0,3,d}$, with $\gamma_i\in H^3(Y)$, would be $8=6+4d$, which is impossible for integral $d$. Thus $H^3(Y)$ lies in the zero factor and $\kappa_\tau$ vanishes on it. Cai's even zero block has rank two with a single nonzero nilpotent entry, while the other two even factors have rank one. This gives $\dim_\KK E^Y_{\ev,0}=12,$ $\rho^Y_{\ev,0}=2,$ and $\gamma^Y_{\ev,0}=1.$

For the exact lattice, after the constant change of basis putting the
leading term in Jordan form, Cai uses a separating matrix $A=I+O(u),$
with coefficients in $\CC[Q^{1/2},Q^{-1/2}]$ and inverse of the same form. After specialization at $Q_*$ it therefore preserves the original formal lattice. By uniqueness in Lemma~\ref{lem:separation}, its zero block is our separated zero factor.

In Cai's frame, this block satisfies 
$$u^2\partial_u\widetilde S = \left[
 \begin{pmatrix}0&2\\0&0\end{pmatrix}
 +u\begin{pmatrix}-19/18&0\\0&19/18\end{pmatrix}
 +u^2\begin{pmatrix}0&*\\-8/81&0\end{pmatrix}
 +O(u^3)
 \right]\widetilde S.$$
Hence, in the notation of Lemma~\ref{lem:residue-lattice}, $n=2,$ $a_{11}=\frac{19}{18},$ $a_{22}=-\frac{19}{18},$ $a_{21}=0,$ $b_{21}=\frac{8}{81}.$ Thus $L^\sharp$ is logarithmic, and 
$$ R_Y^\sharp=
 \begin{pmatrix}19/18&-2\\8/81&-1/18\end{pmatrix},\qquad
 \det(tI-R_Y^\sharp)=t^2-t+5/36.$$

\end{proof}

For later use with $\PP^1$, let $J_\pm$ be its two small-quantum line
factors. In the basis $(1,h)$, where $h=c_1(\cO_{\PP^1}(1))$, put
\[
 M_h=\begin{pmatrix}0&Q'_*\\1&0\end{pmatrix},\qquad
 P_\pm=\tfrac12(I\pm M_h/\sqrt{Q'_*}),\qquad
 \Gr_{\PP^1}=\operatorname{diag}(-1/2,1/2).
\]
After removing the scalar leading term, the residue of $J_\pm$ is $c_\pm=\tr(P_\pm\Gr_{\PP^1})=0.$
Indeed, a regular separating change-of-basis changes the coefficient of $u^{-1}$ by a commutator with the semisimple leading endomorphism, whose diagonal blocks vanish. Hence the residue on each rank-one factor is the corresponding diagonal block of $\Gr_{\PP^1}$, whose trace is $\tr(P_\pm\Gr_{\PP^1})$.

More generally, let $L$ be a rank-two primary lattice and $J$ a rank-one primary lattice with normalized residue $c$. After removing the sum of the two scalar leading terms, the leading nilpotent on $L\otimes J$ is $N_{L\otimes J}=N_L\otimes I.$ As a result,

$$(L\otimes J)^\sharp=L^\sharp\otimes J,\qquad
 R^\sharp_{L\otimes J}=R^\sharp_L+cI,$$ 
whenever $L^\sharp$ is logarithmic.

Also note that by \cite[Theorem~9]{Behrend}, at $T_i=0$, we have
$$ \kappa^{A\times\PP^1}_0=\kappa^A_0\otimes I+I\otimes\kappa^{\PP^1}_0,
 \qquad
 \Gr_{A\times\PP^1}=\Gr_A\otimes I+I\otimes\Gr_{\PP^1}.$$

\begin{proposition}\label{prop:cubic-product}
For every smooth cubic threefold $Y$,
\[n_{\mathrm{cub}}(Y)=1,\qquad n_{\mathrm{cub}}(Y\times\PP^1)=2.\]
\end{proposition}
\begin{proof}
Lemma~\ref{lem:cubic} and Corollary~\ref{Cor:local-count} give $n_{\mathrm{cub}}(Y)=1$. For $Y\times\PP^1$, choose nonzero sufficiently small $Q_*,Q'_*$ with $|Q_*|\ne|Q'_*|$. By the discussion right above, the six eigenvalues are $ \pm2\sqrt{Q'_*}, \pm6\sqrt{3Q_*}\pm2\sqrt{Q'_*},$ and are pairwise distinct. There are only two factors have full rank 12 and Hodge-class rank 2, namely the tensor products of the cubic zero factor with $J_+$ and $J_-$. Their leading nilpotents are nonzero, and their residue polynomial are precisely $P_{\mathrm{cub}}$. The other four factors have rank one. By Corollary~\ref{Cor:local-count} again, we have $n_{\mathrm{cub}}(Y\times\PP^1)=2$.
\end{proof}

\begin{lemma}\label{lem:curve-residue}
Let $C$ be a smooth projective curve of genus $g>1$. On its Hodge-class primary lattice, $L^\sharp$ is logarithmic and $\det(tI-R_C^\sharp)=(t-1/2)^2.$ 
\end{lemma}

\begin{proof}
    Put $\eta=[\mathrm{pt}]$. Since $C$ has no nonconstant genus-zero stable maps, its quantum product is the usual product. After removing the unit term, the connection in the frame $(\eta,1)$ is 
    $$ \nabla^0_{\partial_u}  =\partial_u-\frac1{u^2}  \begin{pmatrix}0&2-2g\\0&0\end{pmatrix}  +\frac1u\begin{pmatrix}1/2&0\\0&-1/2\end{pmatrix}.$$
    Thus $\ell=\KK\eta$ and $L^\sharp=\KK[[u]]\eta\oplus u\KK[[u]]1$. In the frame $(\eta,u1)$, the residue is 
    $$ R_C^\sharp=\begin{pmatrix}1/2&2g-2\\0&1/2\end{pmatrix},$$
    which has characteristic polynomial $(t-1/2)^2$. Compare \cite[Proposition~7]{Cai}.
\end{proof}

\begin{lemma}\label{lem:curve-vanishing}
For every smooth projective curve $C$,
$n_{\mathrm{cub}}(C)=n_{\mathrm{cub}}(C\times\PP^1)=0$.
\end{lemma}

\begin{proof}
We may assume that $C$ is connected. If $g(C)\leq1$, the full cohomological ranks of $C$ and $C\times\PP^1$ are at most four and eight, respectively, so neither has a rank-twelve factor.

Suppose that $g(C)>1$. The unique primary factor of $C$ has Hodge-class rank two, and its specified logarithmic lattice has residue polynomial $(t-1/2)^2$ by Lemma~\ref{lem:curve-residue}. This differs from $P_{\mathrm{cub}}$.

At a nonzero small-quantum point, the two primary factors of $C\times\PP^1$ are obtained by tensoring the curve factor with $J_+$ and $J_-$. But they both have the same residue polynomial $(t-1/2)^2$. Corollary~\ref{Cor:local-count} proves our claims.
\end{proof}

\begin{proposition}\label{prop:surface}
For every smooth projective variety $S$ of dimension at most two, $n_{\mathrm{cub}}(S)=0$.
\end{proposition}

\begin{proof}
By additivity over connected components, we may assume that $S$ is connected. Points have rank one, and curves are covered by Lemma~\ref{lem:curve-vanishing}.

Suppose that $S$ is a surface. Blowing up a point does not change $n_{\mathrm{cub}}$ by Corollary~\ref{cor:blowup}, since $n_{\mathrm{cub}}(\mathrm{pt})=0$. We may therefore pass to a minimal model. By the classification of minimal surfaces  \cite[Corollary~D.4.4]{Reid}, it is either $\PP^2$, a ruled surface, or has nef canonical class.

For $\PP^2$, the full cohomology has rank three, so $n_{\mathrm{cub}}(\PP^2)=0$. A ruled surface is birational to $C\times\PP^1$ for some smooth projective curve $C$. A weak factorization between smooth projective surfaces uses only point blowups and blowdowns, so Corollary~\ref{cor:blowup} and Lemma~\ref{lem:curve-vanishing} give vanishing for every ruled surface.

Finally, suppose that $K_S$ is nef. By \cite[Example~13 and Proposition~50]{Guere}, the big quantum operator $\kappa_\tau^S-T_0I$ is nilpotent; compare the degree argument in \cite[proof of Lemma~5.18]{KKPY}. Hence the reduced spectrum consists of a single eigenvalue and there is only one primary factor. Its Hodge-class part contains $1,H,[\mathrm{pt}]$ for any ample divisor class $H$, and therefore has Hodge-class rank at least three. It is consequently not counted by Definition~\ref{def:nu}. Thus $n_{\mathrm{cub}}(S)=0.$
\end{proof}

\begin{corollary}\label{cor:birational-invariance}
The count $n_{\mathrm{cub}}$ is a birational invariant of smooth projective complex varieties of dimension at most four. It vanishes on rational varieties of these dimensions.
\end{corollary}

\begin{proof}
    By weak factorization \cite[Theorem~0.1.1]{AKMW}, a birational map between smooth projective varieties factors into blowups and blowdowns along smooth centers. After discarding blowups along divisors, every center has codimension at least two, hence dimension at most two. Its count vanishes by Proposition~\ref{prop:surface}, so Corollary~\ref{cor:blowup} shows that every step preserves $n_{\mathrm{cub}}$. 

    Finally, $\dim_\QQ H^*(\PP^d,\QQ)=d+1<12$ for $d\leq4$, so $n_{\mathrm{cub}}(\PP^d)=0$. This proves the assertion for rational varieties.
\end{proof}

\begin{proof}[Proof of Theorem~\ref{thm:main}]
By Proposition~\ref{prop:cubic-product}, $n_{\mathrm{cub}}(Y\times\PP^1)=2$. If $Y\times\PP^1$ were rational, its count would vanish by Corollary~\ref{cor:birational-invariance}, a contradiction.
\end{proof}

\begin{remark}\label{rem:dimension-bound}
The dimension bound in Corollary~\ref{cor:birational-invariance} is sharp. Embed a smooth cubic threefold $Y\subset\PP^4\subset\PP^5$ and put $W=\Bl_Y\PP^5$. Then $W$ is a smooth rational fivefold, whereas 
$$ n_{\mathrm{cub}}(W)  =n_{\mathrm{cub}}(\PP^5)+n_{\mathrm{cub}}(Y)=1$$
by Corollary~\ref{cor:blowup} and Proposition~\ref{prop:cubic-product}. Thus $n_{\mathrm{cub}}$ is not a birational invariant in dimension five.
\end{remark}

\section{Applications}\label{sec:applications}

\begin{proof}[Proof of Corollary~\ref{cor:fiber product}]
Let $l\geq1$ be the largest integer for which $Y\times\PP^l$ is not rational. Such an integer exists by stable rationality; once $Y\times\PP^r$ is rational, so is $Y\times\PP^{r'}$ for every $r'>r$. Set $X=Y\times\PP^l$ and write $n=\dim Y$.

For $m\geq2$, one has $lm>l$. Since $(\PP^l)^m\overset{\mathrm{bir}}{\sim}\PP^{lm}$, we may successively use the rationality of $Y\times\PP^r$ for $r>l$ to obtain
\[X^m\overset{\mathrm{bir}}{\sim}Y^m\times\PP^{lm} \overset{\mathrm{bir}}{\sim}Y^{m-1}\times\PP^{lm+n} \overset{\mathrm{bir}}{\sim}\cdots \overset{\mathrm{bir}}{\sim}\PP^{m(l+n)}.\]
Thus $X^m$ is rational for every $m\geq2$, while $X$ is not rational
by construction.
\end{proof}

\begin{proof}[Proof of Corollary~\ref{cor:applications}]
Choose a smooth stably rational cubic threefold $Y$ by \cite[Theorem~1.3]{TZ}. By Theorem~\ref{thm:main} and Corollary~\ref{cor:fiber product}, there is an integer $l\geq1$ such that $X=Y\times\PP^l$ is not rational and $X^m$ is rational for every $m\geq2$. By the choice of $l$, the variety $Y\times\PP^a$ is rational whenever $a>l$. The identities
$$\Sym^2(Y\times\PP^l) \overset{\mathrm{bir}}{\sim}\Sym^2Y\times\PP^{2l}  \overset{\mathrm{bir}}{\sim}Y\times\PP^{2l+3}$$
follow from \cite[Lemma~2.4]{kol18} and \cite[Exercise~3.10]{KollarCubicNotes}, respectively. Since $2l+3>l$, this proves that $\Sym^2X$ is rational. $X$ is clearly Fano and stably rational because $Y$ is.
\end{proof}

\begin{corollary}\label{cor:fano-fourfold}
For every smooth cubic threefold $Y\subset\PP^4_{\CC}$, the product
$Y\times\PP^1$ is a smooth irrational Fano fourfold. Moreover,
\[\rho(Y\times\PP^1)=2,\qquad\iota(Y\times\PP^1)=2,\qquad (-K_{Y\times\PP^1})^4=192,\]
where $\rho$ denotes the Picard rank and $\iota$ the Fano index. Some of these fourfolds are stably rational.
\end{corollary}

\begin{proof}
Put $A=p_1^*H$ and $L=p_2^*c_1(\cO_{\PP^1}(1))$. By adjunction, $-K_{Y\times\PP^1}=2(A+L),$ which is ample. By the Lefschetz hyperplane theorem, $\Pic(Y)=\ZZ H$, and the projective bundle formula gives $ \Pic(Y\times\PP^1)=\ZZ A\oplus\ZZ L.$ Thus $\rho(Y\times\PP^1)=2$. Since $A+L$ is primitive in the Picard group, the Fano index is exactly two. Finally, using $H^3=3$ and $L^2=0$,
\[(-K_{Y\times\PP^1})^4=16(A+L)^4=64A^3L=192.\]
Irrationality follows from Theorem~\ref{thm:main}. Choosing $Y$ as in \cite[Theorem~1.3]{TZ} gives the stably rational examples.
\end{proof}

\begin{remark}[The $V_{14}$ case]\label{rem:v14}
Let $X=V_{14}$ be a smooth Fano threefold of Picard rank one, index one, and genus eight. Its associated Pfaffian cubic threefold $Y$ is smooth by \cite[Lemma~2.1 and Proposition~A.4]{KuznetsovV14}, and is birational to $X$ by \cite[Remark~2.19]{KuznetsovV14}. Hence $X\times\PP^1 \overset{\mathrm{bir}}{\sim} Y\times\PP^1,$ so $X\times\PP^1$ is not rational by Theorem~\ref{thm:main}.
\end{remark}

\begingroup
\emergencystretch=2em
\printbibliography
\endgroup
\end{document}